\documentclass[12pt]{article}
\usepackage{amsmath,amssymb,amsfonts,amsthm,mdframed,graphicx}
\usepackage{color, colortbl}
\usepackage{xcolor}
\usepackage{array}
\usepackage{multirow}

\usepackage{tikz}
\usepackage[colorlinks,
linkcolor=blue,
anchorcolor=blue,
citecolor=blue
]{hyperref}

\newtheorem{thm}{Theorem}[section]

\newtheorem{remark}[thm]{Remark}

\definecolor{blue}{rgb}{0,0,1}
{}

\begin{document}

\begin{center}
{\large \bf  On partially ordered patterns of length 4 and 5 in permutations}
\end{center}

\begin{center}
Alice L.L. Gao$^{1}$ and 
Sergey Kitaev$^{2}$\\[6pt]

$^{1}$Department of Applied Mathematics\\
Northwestern Polytechnical University,
Xi$^,$an, Shaanxi 710072, P.R. China\\[6pt]

$^{2}$Department of Computer and Information Sciences \\
University of Strathclyde, 26 Richmond Street, Glasgow G1 1XH, UK\\[6pt]

Email: $^{1}${\tt llgao@nwpu.edu.cn},
	   $^{2}${\tt sergey.kitaev@cis.strath.ac.uk}
	   \end{center}

\noindent\textbf{Abstract.}
Partially ordered patterns (POPs) generalize the notion of classical patterns studied widely in the literature in the context of permutations, words, compositions and partitions. In an occurrence of a POP, the relative order of some of the elements is not important. Thus, any POP of length $k$ is defined by a partially ordered set on $k$ elements, and classical patterns correspond to $k$-element chains. The notion of a POP provides  a convenient language to deal with larger sets of permutation patterns.

This paper contributes to a long line of research on classical permutation patterns of length 4 and 5, and beyond, by conducting a systematic search of connections between sequences in the Online Encyclopedia of Integer Sequences (OEIS) and permutations avoiding POPs of length 4 and 5. As the result, we (i) obtain  13 new enumerative results for classical patterns of length 4 and 5, and a number of results for patterns of arbitrary length, (ii) collect under one roof many sporadic results in the literature related to avoidance of patterns of length 4 and 5, and (iii) conjecture 6 connections to the OEIS. Among the most intriguing bijective questions we state, 7 are related to explaining Wilf-equivalence of various sets of patterns, e.g.\ 5 or 8 patterns of length 4, and 2 or 6 patterns of length~5. \\



\section{Introduction}\label{intro-sec}

An occurrence of a (classical) permutation pattern $p=p_1\cdots p_k$ in a permutation $\pi=\pi_1\cdots\pi_n$ is a subsequence $\pi_{i_1}\cdots\pi_{i_k}$, where $1\leq i_1<\cdots< i_k\leq n$, such that $\pi_{i_j}<\pi_{i_m}$ if and only if $p_j<p_m$. For example, the permutation $31425$ has three occurrences of the pattern 123, namely, the subsequences 345, 145, and 125. Permutation patterns are the subject of lots of interest in the literature (e.g.\ see \cite{Kit5} and references therein). 

A {\em partially ordered pattern} ({\em POP}) $p$ of length $k$ is defined by a $k$-element partially ordered set (poset) $P$ labeled by the elements in $\{1,\ldots,k\}$. An occurrence of such a POP $p$ in a permutation $\pi=\pi_1\cdots\pi_n$ is a subsequence $\pi_{i_1}\cdots\pi_{i_k}$, where $1\leq i_1<\cdots< i_k\leq n$,  such that $\pi_{i_j}<\pi_{i_m}$ if and only if $j<m$ in $P$. Thus, a classical pattern of length $k$ corresponds to a $k$-element chain. For example, the POP $p=$ \hspace{-3.5mm}
\begin{minipage}[c]{3.5em}\scalebox{1}{
\begin{tikzpicture}[scale=0.5]

\draw [line width=1](0,-0.5)--(0,0.5);

\draw (0,-0.5) node [scale=0.4, circle, draw,fill=black]{};
\draw (1,-0.5) node [scale=0.4, circle, draw,fill=black]{};
\draw (0,0.5) node [scale=0.4, circle, draw,fill=black]{};

\node [left] at (0,-0.6){${\small 3}$};
\node [right] at (1,-0.6){${\small 2}$};
\node [left] at (0,0.6){${\small 1}$};

\end{tikzpicture}
}\end{minipage}
occurs five times in the permutation 41523, namely, as the subsequences 412, 413, 452, 453, and 523. Clearly, avoiding $p$ is the same as avoiding the patterns 312, 321	and 231 at the same time. 

POPs can also be defined using one-line notation by providing the minimal set of relations defining the respective poset. For example, the POP in the example above can be defined by $\{1>3\}$, while the POP in  Theorem~\ref{thm-11} can be defined by $\{1>3, 1>2, 4>2\}$. 

POPs were introduced in \cite{Kit1}, and they were studied in the context of permutations, words and compositions in \cite{HKM07,Kit2,Kit3,Kit4,KM03,KP10}. The notion of a POP provides a uniform notation for several combinatorial structures such as peaks, valleys, modified maxima and minima, $p$-descents in permutations, and many others \cite{Kit4}. Moreover, POPs provide a convenient language to deal with larger sets of permutation patterns. Thus, by noticing connections to POPs, in this paper we collect under one roof many sporadic results in the literature related to the avoidance of patterns of length 4 and 5. For example, the simultaneous avoidance of the patterns 3214, 3124, 2134, and  2143 considered, up to trivial bijections, in \cite{D18} is nothing else but the avoidance of the POP \begin{minipage}[c]{3em}\scalebox{1}{
\begin{tikzpicture}[scale=0.3]
\draw [line width=1](0,0)--(0,1)--(0,2);
\draw (0,0) node [scale=0.3, circle, draw,fill=black]{};
\draw (0,1) node [scale=0.3, circle, draw,fill=black]{};
\draw (0,2) node [scale=0.3, circle, draw,fill=black]{};
\draw (1,0) node [scale=0.3, circle, draw,fill=black]{};
\node [left] at (0,-0.1){\small$2$};
\node [left] at (0,1){\small$1$};
\node [left] at (0,2.1){\small$4$};
\node [right] at (1,-0.1){\small$3$};
\end{tikzpicture}
}\end{minipage}
, which suggests natural directions of research to study the avoidance of the POPs 
\begin{minipage}[c]{3em}\scalebox{1}{
\begin{tikzpicture}[scale=0.3]
\draw [line width=1](0,0)--(0,1)--(0,2);
\draw (0,0) node [scale=0.3, circle, draw,fill=black]{};
\draw (0,1) node [scale=0.3, circle, draw,fill=black]{};
\draw (0,2) node [scale=0.3, circle, draw,fill=black]{};
\draw (1,0) node [scale=0.3, circle, draw,fill=black]{};
\node [left] at (0,-0.1){\small$2$};
\node [left] at (0,1){\small$4$};
\node [left] at (0,2.1){\small$1$};
\node [right] at (1,-0.1){\small$3$};
\end{tikzpicture}
}\end{minipage}
,
\begin{minipage}[c]{3em}\scalebox{1}{
\begin{tikzpicture}[scale=0.3]
\draw [line width=1](0,0)--(0,1)--(0,2);
\draw (0,0) node [scale=0.3, circle, draw,fill=black]{};
\draw (0,1) node [scale=0.3, circle, draw,fill=black]{};
\draw (0,2) node [scale=0.3, circle, draw,fill=black]{};
\draw (1,0) node [scale=0.3, circle, draw,fill=black]{};
\node [left] at (0,-0.1){\small$4$};
\node [left] at (0,1){\small$1$};
\node [left] at (0,2.1){\small$2$};
\node [right] at (1,-0.1){\small$3$};
\end{tikzpicture}
}\end{minipage}
, etc.

In any case, the starting point in our project was utilization of the software produced by Stephen Gardiner in 2018 as part of his MSc studies at the University of Strathclyde. The software is able to go exhaustively through all POPs of length 4 and 5 (length 3 POPs are rather trivial in our context and were omitted) and detect any connections to the Online Encyclopedia of Integer Sequences  (OEIS) \cite{oeis}. So, our original goal was to explore the variety of objects in the OEIS that are equinumerous to length 4, 5 POP-avoiding permutations, and to justify any observations, which often required non-trivial enumeration or a bijection, but sometimes were given ``for free'' via exactly the same pattern avoidance studied previously.  Some of our results for  length 4, 5 POP-avoiding permutations follow from more general theorems we prove.

We ended up with observing connections to 38 sequences in the OEIS, out of which 18 sequences have no known interpretation in terms of pattern avoidance. We justified all but 6 connections all related to POPs of length 5.  Also, in our studies, we obtain 13 new enumerations for pattern avoiding permutations for patterns of length 4 and 5, in particular, contributing to a long line of enumerative results on length 4 patterns, e.g. with enumeration of triples of such patterns being concluded in \cite{CMS19}.  

Our results can be found in Tables~\ref{tab-pop-4-easy}--\ref{tab-pop-5}, and our conjectures in  Table~\ref{tab-pop-5-conj}. Also, in Section~\ref{general-sec} we give a number of general results. Note that all these results can give new results for many more patterns/POPs using Theorem~\ref{trivial-sym-thm}, which discusses equivalence modulo the complement of poset labels, or reflecting a poset with respect to a horizontal line. However, we would like to stress that the goal of this paper is not in achieving any classifications (which is an interesting direction, of course), instead considering  just a single POP corresponding to a sequence in the OEIS.

\begin{table}[t]
\begin{center}
\begin{tabular}{c|l|c|c}
\hline
\rowcolor{gray!20!}
{\bf POP} & {\bf Sequence (beginning with $n=1$)} &  {\bf OEIS} &  {\bf Ref} \\ 
\hline
\begin{tikzpicture}[scale=0.3]
\draw [line width=1](0,0)--(0,1);
\draw (0,0) node [scale=0.3, circle, draw,fill=black]{};
\draw (0,1) node [scale=0.3, circle, draw,fill=black]{};
\draw (1,0) node [scale=0.3, circle, draw,fill=black]{};
\draw (2,0) node [scale=0.3, circle, draw,fill=black]{};
\node [below] at (0,0){\small$2$};
\node [above] at (0,1){\small$1$};
\node [below] at (1,0){\small$3$};
\node [below] at (2,0){\small$4$};
\end{tikzpicture}
&  1, 2, 6, 12, 20, 30, 42, 56, 72, ...
 & \cellcolor{gray!20!}
 A103505 & Thm~\ref{thm-19} \\ 
\hline
\begin{tikzpicture}[scale=0.3]
\draw [line width=1](0,0)--(0,1);
\draw (0,0) node [scale=0.3, circle, draw,fill=black]{};
\draw (0,1) node [scale=0.3, circle, draw,fill=black]{};
\draw (1,0) node [scale=0.3, circle, draw,fill=black]{};
\draw (2,0) node [scale=0.3, circle, draw,fill=black]{};
\node [below] at (0,0){\small$3$};
\node [above] at (0,1){\small$1$};
\node [below] at (1,0){\small$2$};
\node [below] at (2,0){\small$4$};
\end{tikzpicture}
&  1, 2, 6, 12, 25, 48, 91, 168, 306, ... &  \cellcolor{gray!20!} A045925 &  Thm~\ref{thm-22} \\ 
\hline
\multirow{2}{*}{ 
\begin{tikzpicture}[scale=0.3]
\draw [line width=1](0,0)--(1,1)--(2,0);
\draw (0,0) node [scale=0.3, circle, draw,fill=black]{};
\draw (1,1) node [scale=0.3, circle, draw,fill=black]{};
\draw (2,0) node [scale=0.3, circle, draw,fill=black]{};
\draw (3,0) node [scale=0.3, circle, draw,fill=black]{};
\node [left] at (0,-0.2){\small$2$};
\node [right] at (1,1.2){\small$1$};
\node [left] at (2,-0.2){\small$3$};
\node [right] at (3,-0.2){\small$4$};
\end{tikzpicture}
}
& 
 &  A129952 & \\
 & 1, 2, 6, 16, 40, 96, 224, 512, 1152, ... &  \cellcolor{gray!20!}  A057711 &  Thm~\ref{thm-16} \\ 
\hline
\begin{tikzpicture}[scale=0.3]
\draw [line width=1](0,0)--(1,1)--(2,0);
\draw [line width=1](1,1)--(1,0);
\draw (0,0) node [scale=0.3, circle, draw,fill=black]{};
\draw (1,1) node [scale=0.3, circle, draw,fill=black]{};
\draw (2,0) node [scale=0.3, circle, draw,fill=black]{};
\draw (1,0) node [scale=0.3, circle, draw,fill=black]{};
\node [below] at (0,0){\small$2$};
\node [right] at (1,1.2){\small$1$};
\node [below] at (2,0){\small$4$};
\node [below] at (1,0){\small$3$};
\end{tikzpicture}
&1, 2, 6, 18, 54, 162, 486, 1458, 4374,...  & A025192 & Thm~\ref{thm-9}  \\ 
\hline
\begin{tikzpicture}[scale=0.3]
\draw [line width=1](0,0)--(0,1)--(1,0)--(1,1)--(0,0);
\draw (0,0) node [scale=0.3, circle, draw,fill=black]{};
\draw (0,1) node [scale=0.3, circle, draw,fill=black]{};
\draw (1,0) node [scale=0.3, circle, draw,fill=black]{};
\draw (1,1) node [scale=0.3, circle, draw,fill=black]{};
\node [left] at (0,-0.1){\small$2$};
\node [left] at (0,1.1){\small$1$};
\node [right] at (1,-0.1){\small$3$};
\node [right] at (1,1.1){\small$4$};
\end{tikzpicture}
& 1, 2, 6, 20, 68, 232, 792, 2704, 9232, ... &  A006012 & Thm~\ref{thm-10}  \\ 
\hline
\begin{tikzpicture}[scale=0.3]
\draw [line width=1](0,0)--(0,1)--(0,2);
\draw (0,0) node [scale=0.3, circle, draw,fill=black]{};
\draw (0,1) node [scale=0.3, circle, draw,fill=black]{};
\draw (0,2) node [scale=0.3, circle, draw,fill=black]{};
\draw (1,0) node [scale=0.3, circle, draw,fill=black]{};
\node [left] at (0,-0.1){\small$2$};
\node [left] at (0,1){\small$1$};
\node [left] at (0,2.1){\small$3$};
\node [right] at (1,-0.1){\small$4$};
\end{tikzpicture}
&1, 2, 6, 20, 70, 252, 924, 3432, 12870,...  & A000984 & Thm~\ref{thm-4}  \\ 
\hline
\end{tabular}
\end{center}
\caption{POPs of length 4 that are particular cases in our general theorems. For the highlighted OEIS sequences no interpretation in terms of permutation patterns was known until this work. The connections to permutation patterns in A129952, A025192, A006012, A000984 are via \cite{BurKit},  \cite{BLNPPRT16},  \cite{B17}, \cite{D18}, respectively.}\label{tab-pop-4-easy}
\end{table}


\begin{table}[!ht]
\begin{center}
\begin{tabular}{c|l|c|c}
\hline
\rowcolor{gray!20!}
{\bf POP} & {\bf Sequence (beginning with $n=1$)} &  {\bf OEIS} &  {\bf Ref} \\ 
\hline
\multirow{2}{*}{ 
\begin{tikzpicture}[scale=0.3]
\draw [line width=1](0,0)--(0,1);
\draw (0,0) node [scale=0.3, circle, draw,fill=black]{};
\draw (0,1) node [scale=0.3, circle, draw,fill=black]{};
\draw (1,0) node [scale=0.3, circle, draw,fill=black]{};
\draw (2,0) node [scale=0.3, circle, draw,fill=black]{};
\node [below] at (0,0){\small$4$};
\node [left] at (0,1.2){\small$1$};
\node [below] at (1,0){\small$2$};
\node [below] at (2,0){\small$3$};
\end{tikzpicture}
}
&  &  \cellcolor{gray!20!} A214663  &   \\ 
&1, 2, 6, 12, 25, 57, 124, 268, 588,... & \cellcolor{gray!20!} A232164 & Thm~\ref{thm-24} \\
\hline
\begin{tikzpicture}[scale=0.3]
\draw [line width=1](0,0)--(0,1);
\draw [line width=1](1,0)--(1,1);
\draw (0,0) node [scale=0.3, circle, draw,fill=black]{};
\draw (0,1) node [scale=0.3, circle, draw,fill=black]{};
\draw (1,0) node [scale=0.3, circle, draw,fill=black]{};
\draw (1,1) node [scale=0.3, circle, draw,fill=black]{};
\node [left] at (0,-0.2){\small$2$};
\node [left] at (0,1.2){\small$1$};
\node [right] at (1,-0.2){\small$3$};
\node [right] at (1,1.2){\small$4$};
\end{tikzpicture}
&1, 2, 6, 18, 50, 130, 322, 770, 1794,...  & \cellcolor{gray!20!}  A048495 & Thm~\ref{thm-18} \\ 
\hline
\begin{tikzpicture}[scale=0.3]
\draw [line width=1](0,0)--(0,1);
\draw [line width=1](1,0)--(1,1);
\draw (0,0) node [scale=0.3, circle, draw,fill=black]{};
\draw (0,1) node [scale=0.3, circle, draw,fill=black]{};
\draw (1,0) node [scale=0.3, circle, draw,fill=black]{};
\draw (1,1) node [scale=0.3, circle, draw,fill=black]{};
\node [left] at (0,-0.2){\small$3$};
\node [left] at (0,1.2){\small$1$};
\node [right] at (1,-0.2){\small$2$};
\node [right] at (1,1.2){\small$4$};
\end{tikzpicture}
&1, 2, 6, 18, 52, 152, 444, 1296, 3784,...  &\cellcolor{gray!20!}  A077835 & Thm~\ref{thm-21}  \\ 
\hline
\begin{tikzpicture}[scale=0.3]
\draw [line width=1](0,0)--(0,1);
\draw [line width=1](1,0)--(1,1);
\draw (0,0) node [scale=0.3, circle, draw,fill=black]{};
\draw (0,1) node [scale=0.3, circle, draw,fill=black]{};
\draw (1,0) node [scale=0.3, circle, draw,fill=black]{};
\draw (1,1) node [scale=0.3, circle, draw,fill=black]{};
\node [left] at (0,-0.2){\small$4$};
\node [left] at (0,1.2){\small$1$};
\node [right] at (1,-0.2){\small$2$};
\node [right] at (1,1.2){\small$3$};
\end{tikzpicture}
 &1, 2, 6, 18, 50, 134, 358, 962, 2594,...  &\cellcolor{gray!20!}  A271897 & Thm~\ref{thm-23}\\ 
\hline
\begin{tikzpicture}[scale=0.3]
\draw [line width=1](0,0)--(1,1)--(2,0);
\draw (0,0) node [scale=0.3, circle, draw,fill=black]{};
\draw (1,1) node [scale=0.3, circle, draw,fill=black]{};
\draw (2,0) node [scale=0.3, circle, draw,fill=black]{};
\draw (3,0) node [scale=0.3, circle, draw,fill=black]{};
\node [left] at (0,-0.2){\small$2$};
\node [right] at (1,1.2){\small$1$};
\node [left] at (2,-0.2){\small$4$};
\node [right] at (3,-0.2){\small$3$};
\end{tikzpicture}
&1, 2, 6, 16, 40, 100, 252, 636, 1604,...  &  A111281&  Thm~\ref{thm-15} \\ 
\hline
\begin{tikzpicture}[scale=0.3]
\draw [line width=1](0,0)--(1,1)--(2,0);
\draw (0,0) node [scale=0.3, circle, draw,fill=black]{};
\draw (1,1) node [scale=0.3, circle, draw,fill=black]{};
\draw (2,0) node [scale=0.3, circle, draw,fill=black]{};
\draw (3,0) node [scale=0.3, circle, draw,fill=black]{};
\node [left] at (0,-0.2){\small$3$};
\node [right] at (1,1.2){\small$1$};
\node [left] at (2,-0.2){\small$4$};
\node [right] at (3,-0.2){\small$2$};
\end{tikzpicture}
&1, 2, 6, 16, 44, 120, 328, 896, 2448,...  &  A002605&  Thm~\ref{thm-20} \\ 
\hline
\begin{tikzpicture}[scale=0.3]
\draw [line width=1](0,0)--(1,1)--(2,0);
\draw (0,0) node [scale=0.3, circle, draw,fill=black]{};
\draw (1,1) node [scale=0.3, circle, draw,fill=black]{};
\draw (2,0) node [scale=0.3, circle, draw,fill=black]{};
\draw (3,0) node [scale=0.3, circle, draw,fill=black]{};
\node [left] at (0,-0.2){\small$1$};
\node [right] at (1,1.2){\small$2$};
\node [left] at (2,-0.2){\small$4$};
\node [right] at (3,-0.2){\small$3$};
\end{tikzpicture}
&1, 2, 6, 16, 42, 110, 288, 754, 1974,...  &  A111282& Thm~\ref{thm-17} \\ 
\hline
\begin{tikzpicture}[scale=0.3]
\draw [line width=1](0,0)--(0,1)--(1,0)--(1,1);
\draw (0,0) node [scale=0.3, circle, draw,fill=black]{};
\draw (0,1) node [scale=0.3, circle, draw,fill=black]{};
\draw (1,0) node [scale=0.3, circle, draw,fill=black]{};
\draw (1,1) node [scale=0.3, circle, draw,fill=black]{};
\node [left] at (0,-0.1){\small$2$};
\node [left] at (0,1.1){\small$1$};
\node [right] at (1,-0.1){\small$3$};
\node [right] at (1,1.1){\small$4$};
\end{tikzpicture}
&1, 2, 6, 19, 59, 180, 544, 1637, 4917,...  &A111277  & Thm~\ref{thm-12} \\ 
\hline
\multirow{2}{*}{ 
\begin{tikzpicture}[scale=0.3]
\draw [line width=1](0,0)--(0,1)--(1,0)--(1,1);
\draw (0,0) node [scale=0.3, circle, draw,fill=black]{};
\draw (0,1) node [scale=0.3, circle, draw,fill=black]{};
\draw (1,0) node [scale=0.3, circle, draw,fill=black]{};
\draw (1,1) node [scale=0.3, circle, draw,fill=black]{};
\node [left] at (0,-0.1){\small$3$};
\node [left] at (0,1.1){\small$1$};
\node [right] at (1,-0.1){\small$2$};
\node [right] at (1,1.1){\small$4$};
\end{tikzpicture}
}
& & \cellcolor{gray!20!}  A052544 & \\
&1, 2, 6, 19, 60, 189, 595, 1873, 5896,...  &\cellcolor{gray!20!}  A204200  & Thm~\ref{thm-11} \\ 
\hline
\end{tabular}
\end{center}
\caption{POPs of length 4 with longest chain of size 2.  For the highlighted OEIS sequences no interpretation in terms of permutation patterns was known until this work. Connections to A111281, A002605,  A111282, A111277 are via \cite{AAAHHMv05}, \cite{DMS}, \cite{AAAHHMv05},\cite{AAAHHMv05},  respectively.}\label{tab-pop-4}
\end{table}

\begin{table}[!ht]
\begin{center}
\begin{tabular}{c|l|c|c}
\hline
\rowcolor{gray!20!}
{\bf POP} & {\bf Sequence (beginning with $n=1$)} &  {\bf OEIS} &  {\bf Ref} \\ 
\hline
\begin{tikzpicture}[scale=0.3]
\draw [line width=1](0,0)--(0,1)--(0,2);
\draw (0,0) node [scale=0.3, circle, draw,fill=black]{};
\draw (0,1) node [scale=0.3, circle, draw,fill=black]{};
\draw (0,2) node [scale=0.3, circle, draw,fill=black]{};
\draw (1,0) node [scale=0.3, circle, draw,fill=black]{};
\node [left] at (0,-0.1){\small$2$};
\node [left] at (0,1){\small$1$};
\node [left] at (0,2.1){\small$4$};
\node [right] at (1,-0.1){\small$3$};
\end{tikzpicture}
&1, 2, 6, 20, 71, 264, 1015, 4002, 16094,...  &A049124  &  Thm~\ref{thm-14}\\ 
\hline
\begin{tikzpicture}[scale=0.3]
\draw [line width=1]
(0,0)--(0,1)--(1,2)--(2,1);
\draw (0,0) node [scale=0.3, circle, draw,fill=black]{};
\draw (0,1) node [scale=0.3, circle, draw,fill=black]{};
\draw (1,2) node [scale=0.3, circle, draw,fill=black]{};
\draw (2,1) node [scale=0.3, circle, draw,fill=black]{};
\node [left] at (0,-0.1){\small$4$};
\node [left] at (0,1){\small $2$};
\node [right] at (1,2.1){\small $1$};
\node [right] at (2,1){\small $3$};
\end{tikzpicture}
&1, 2, 6, 21, 80, 322, 1346, 5783, 25372,...  & A257561 &  Thm~\ref{thm-8}\\ 
\hline
\begin{tikzpicture}[scale=0.3]
\draw [line width=1]
(0,0)--(0,1)--(1,2)--(2,1);
\draw (0,0) node [scale=0.3, circle, draw,fill=black]{};
\draw (0,1) node [scale=0.3, circle, draw,fill=black]{};
\draw (1,2) node [scale=0.3, circle, draw,fill=black]{};
\draw (2,1) node [scale=0.3, circle, draw,fill=black]{};
\node [left] at (0,-0.1){\small$2$};
\node [left] at (0,1){\small $3$};
\node [right] at (1,2.1){\small $1$};
\node [right] at (2,1){\small $4$};
\end{tikzpicture}
&1, 2, 6, 21, 79, 309, 1237, 5026, 20626,...  & A111279  &  Thm~\ref{thm-5}\\ 
\hline
\begin{tikzpicture}[scale=0.3]
\draw [line width=1]
(0,0)--(0,1)--(1,2)--(2,1);
\draw (0,0) node [scale=0.3, circle, draw,fill=black]{};
\draw (0,1) node [scale=0.3, circle, draw,fill=black]{};
\draw (1,2) node [scale=0.3, circle, draw,fill=black]{};
\draw (2,1) node [scale=0.3, circle, draw,fill=black]{};
\node [left] at (0,-0.1){\small$2$};
\node [left] at (0,1){\small $4$};
\node [right] at (1,2.1){\small $1$};
\node [right] at (2,1){\small $3$};
\end{tikzpicture}
&1, 2, 6, 21, 80, 322, 1347, 5798, 25512,...  & A106228 &  Thm~\ref{thm-6}\\ 
\hline
\begin{tikzpicture}[scale=0.3]
\draw [line width=1]
(0,0)--(0,1)--(1,2)--(2,1);
\draw (0,0) node [scale=0.3, circle, draw,fill=black]{};
\draw (0,1) node [scale=0.3, circle, draw,fill=black]{};
\draw (1,2) node [scale=0.3, circle, draw,fill=black]{};
\draw (2,1) node [scale=0.3, circle, draw,fill=black]{};
\node [left] at (0,-0.1){\small$2$};
\node [left] at (0,1){\small $1$};
\node [right] at (1,2.1){\small $3$};
\node [right] at (2,1){\small $4$};
\end{tikzpicture}
&1, 2, 6, 21, 79, 311, 1265, 5275, 22431,...  & A033321 & Thm~\ref{thm-3} \\ 
\hline
\begin{tikzpicture}[scale=0.3]
\draw [line width=1](0,2)--(1,1)--(1,0);
\draw [line width=1](1,1)--(2,2);
\draw (0,2) node [scale=0.3, circle, draw,fill=black]{};
\draw (1,1) node [scale=0.3, circle, draw,fill=black]{};
\draw (1,0) node [scale=0.3, circle, draw,fill=black]{};
\draw (2,2) node [scale=0.3, circle, draw,fill=black]{};
\node [left] at (0,2){{\small $3$}};
\node [right] at (1,0.9){{\small $1$}};
\node [right] at (1,-0.1){{\small $2$}};
\node [right] at (2,2){{\small $4$}};
\end{tikzpicture}
&  1, 2, 6, 22, 90, 394, 1806, 8558, 41586, ... &  A006318 & Thm~\ref{thm-1} \\ 
\hline
 \begin{tikzpicture}[scale=0.3]
\draw [line width=1]
(0,1)--(1,2)--(2,1)--(1,0)--(0,1);
\draw (0,1) node [scale=0.3, circle, draw,fill=black]{};
\draw (1,2) node [scale=0.3, circle, draw,fill=black]{};
\draw (2,1) node [scale=0.3, circle, draw,fill=black]{};
\draw (1,0) node [scale=0.3, circle, draw,fill=black]{};
\node [left] at (0,1){{\small $2$}};
\node [right] at (1.1,2.1){{\small $1$}};
\node [right] at (2,1){{\small $3$}};
\node [right] at (1.1,-0.2){{\small $4$}};
\end{tikzpicture}
& 1, 2, 6, 22, 90, 396, 1837, 8864, 44074,... &A053617  & Thm~\ref{thm-7} \\ 
\hline
\begin{tikzpicture}[scale=0.3]
\draw [line width=1](0,1)--(1,2)--(2,1)--(1,0)--(0,1);
\draw (0,1) node [scale=0.3, circle, draw,fill=black]{};
\draw (1,2) node [scale=0.3, circle, draw,fill=black]{};
\draw (2,1) node [scale=0.3, circle, draw,fill=black]{};
\draw (1,0) node [scale=0.3, circle, draw,fill=black]{};
\node [left] at (0,1){{\small $1$}};
\node [right] at (1.1,2.1){{\small $3$}};
\node [right] at (2,1){{\small $4$}};
\node [right] at (1.1,-0.2){{\small $2$}};
\end{tikzpicture}
&  1, 2, 6, 22, 90, 395, 1823, 8741, 43193, ... &  A165546 & Thm~\ref{thm-2} \\ 
\hline

\end{tabular}
\end{center}
\caption{POPs of length 4 with longest chain of size 3. Avoiding the POPs in this table is trivially equivalent to avoiding the patterns in the sequences. The enumerations for  A049124, A257561,  A111279, A106228, A033321, A006318 are coming from \cite{D18}, \cite{AHPSV18}, \cite{CM18}, \cite{CM18-2}, \cite{BB16}, \cite{Kr03}, respectively. Also, A053617 and A165546 appear in \cite{AHPSV18} and \cite{KS03}, respectively.}\label{tab-pop-4-more}
\end{table}

\begin{table}[!ht]
\begin{center}
\begin{tabular}{c|l|c|c}
\hline
\rowcolor{gray!20!}
{\bf POP} & {\bf Sequence (beginning with $n=1$)} &  {\bf OEIS} &  {\bf Ref} \\ 
\hline
\begin{tikzpicture}[scale=0.4]
\draw [line width=1](0,0)--(0,1);
\draw (0,0) node [scale=0.3, circle, draw,fill=black]{};
\draw (0,1) node [scale=0.3, circle, draw,fill=black]{};
\draw (1,0) node [scale=0.3, circle, draw,fill=black]{};
\draw (2,0) node [scale=0.3, circle, draw,fill=black]{};
\draw (3,0) node [scale=0.3, circle, draw,fill=black]{};
\node [below] at (0,0){\small$5$};
\node [left] at (0,1.2){\small$1$};
\node [below] at (1,0){\small$2$};
\node [below] at (2,0){\small$3$};
\node [below] at (3,0){\small$4$};
\end{tikzpicture}
& 1, 2, 6, 24, 60, 150, 399, 1145,... &  \cellcolor{gray!20!} A276838 &  Thm~\ref{thm-A2}\\ 
\hline
\begin{tikzpicture}[scale=0.4]
\draw [line width=1](0,0)--(0,1);
\draw (0,0) node [scale=0.3, circle, draw,fill=black]{};
\draw (0,1) node [scale=0.3, circle, draw,fill=black]{};
\draw (1,0) node [scale=0.3, circle, draw,fill=black]{};
\draw (2,0) node [scale=0.3, circle, draw,fill=black]{};
\draw (3,0) node [scale=0.3, circle, draw,fill=black]{};
\node [below] at (0,0){\small$2$};
\node [left] at (0,1.2){\small$1$};
\node [below] at (1,0){\small$3$};
\node [below] at (2,0){\small$4$};
\node [below] at (3,0){\small$5$};
\end{tikzpicture}
&1, 2, 6, 24, 60, 120, 210, 336,...  & \cellcolor{gray!20!}   A007531 &  Thm~\ref{thm-A1}\\ 
\hline
\begin{tikzpicture}[scale=0.4]
\draw [line width=1](0,0)--(1.5,1)--(1,0);
\draw [line width=1](1.5,1)--(2,0);
\draw [line width=1](1.5,1)--(3,0);
\draw (0,0) node [scale=0.3, circle, draw,fill=black]{};
\draw (1.5,1) node [scale=0.3, circle, draw,fill=black]{};
\draw (2,0) node [scale=0.3, circle, draw,fill=black]{};
\draw (1,0) node [scale=0.3, circle, draw,fill=black]{};
\draw (3,0) node [scale=0.3, circle, draw,fill=black]{};
\node [below] at (0,0){\small$2$};
\node [right] at (1.5,1.2){\small$1$};
\node [below] at (2,0){\small$4$};
\node [below] at (1,0){\small$3$};
\node [below] at (3,0){\small$5$};
\end{tikzpicture}
&1, 2, 6, 24, 96, 384, 1536, 6144,...  & \cellcolor{gray!20!} A084509 & Thm~\ref{thm-A5} \\ 
\hline
\begin{tikzpicture}[scale=0.4]
\draw [line width=1](0,0)--(1,1)--(1.5,0);
\draw [line width=1](1,1)--(3,0);
\draw [line width=1](0,0)--(2,1)--(1.5,0);
\draw [line width=1](2,1)--(3,0);
\draw (0,0) node [scale=0.3, circle, draw,fill=black]{};
\draw (1.5,0) node [scale=0.3, circle, draw,fill=black]{};
\draw (3,0) node [scale=0.3, circle, draw,fill=black]{};
\draw (1,1) node [scale=0.3, circle, draw,fill=black]{};
\draw (2,1) node [scale=0.3, circle, draw,fill=black]{};
\node [below] at (0,0){\small$2$};
\node [below] at (1.5,0){\small$3$};
\node [below] at (3,0){\small$4$};
\node [above] at (1,1){\small$1$};
\node [above] at (2,1){\small$5$};
\end{tikzpicture}
& 1, 2, 6, 24, 108, 504, 2376, 11232,... & \cellcolor{gray!20!} A094433 & Thm~\ref{thm-A12}  \\ 
\hline
\begin{tikzpicture}[scale=0.4]
\draw [line width=1](0,0)--(0,1)--(1,0)--(1,1)--(0,0);
\draw (0,0) node [scale=0.3, circle, draw,fill=black]{};
\draw (0,1) node [scale=0.3, circle, draw,fill=black]{};
\draw (1,0) node [scale=0.3, circle, draw,fill=black]{};
\draw (1,1) node [scale=0.3, circle, draw,fill=black]{};
\draw (2,0) node [scale=0.3, circle, draw,fill=black]{};
\node [below] at (0,0){\small$2$};
\node [above] at (0,1){\small$1$};
\node [below] at (1,0){\small$3$};
\node [above] at (1,1){\small$4$};
\node [below] at (2,0){\small$5$};
\end{tikzpicture}
& 1, 2, 6, 24, 100, 408, 1624, 6336,... &\cellcolor{gray!20!}  A094012  & Thm~\ref{thm-A10} \\ 
\hline
\begin{tikzpicture}[scale=0.4]
\draw [line width=1](0,0)--(0,1)--(0,2)--(0,3);
\draw (0,0) node [scale=0.3, circle, draw,fill=black]{};
\draw (0,1) node [scale=0.3, circle, draw,fill=black]{};
\draw (1,0) node [scale=0.3, circle, draw,fill=black]{};
\draw (0,2) node [scale=0.3, circle, draw,fill=black]{};
\draw (0,3) node [scale=0.3, circle, draw,fill=black]{};
\node [left] at (0,0){\small$4$};
\node [left] at (0,1){\small$3$};
\node [right] at (1,0){\small$5$};
\node [left] at (0,2){\small$2$};
\node [left] at (0,3){\small$1$};
\end{tikzpicture}
&1, 2, 6, 24, 115, 618, 3591, 22088,...  & A128088  & Thm~\ref{thm-A4} \\ 
\hline
\end{tabular}
\end{center}
\caption{POPs of length 5. For the highlighted OEIS sequences no interpretation in terms of permutation patterns was known until this work. Connection to A128088 is via \cite{BousqMe02-03,Ges1990}. Note that the $a(0)=a(1)=a(2)=0$ in A007531, but the rest is exactly our sequence.}\label{tab-pop-5}
\end{table}

\begin{table}[!ht]
\begin{center}
\begin{tabular}{c|l|c}
\hline
\rowcolor{gray!20!}
{\bf POP} & {\bf Sequence (beginning with $n=1$)} &  {\bf OEIS} 
\\ 
\hline
\begin{tikzpicture}[scale=0.4]
\draw [line width=1](0,0)--(1,1)--(1,2);
\draw [line width=1](1,1)--(2,0);
\draw (0,0) node [scale=0.3, circle, draw,fill=black]{};
\draw (1,1) node [scale=0.3, circle, draw,fill=black]{};
\draw (3,0) node [scale=0.3, circle, draw,fill=black]{};
\draw (2,0) node [scale=0.3, circle, draw,fill=black]{};
\draw (1,2) node [scale=0.3, circle, draw,fill=black]{};
\node [below] at (0,0){\small$2$};
\node [left] at (1,1){\small$1$};
\node [below] at (3,0){\small$3$};
\node [below] at (2,0){\small$4$};
\node [left] at (1,2){\small$5$};
\end{tikzpicture}
&1, 2, 6, 24, 110, 540, 2772, 14704,...  & \cellcolor{gray!20!} A216879  
\\ 
\hline
\begin{tikzpicture}[scale=0.4]
\draw [line width=1](0,0)--(1,1)--(1,2);
\draw [line width=1](1,0)--(1,1)--(2,0);
\draw (0,0) node [scale=0.3, circle, draw,fill=black]{};
\draw (1,1) node [scale=0.3, circle, draw,fill=black]{};
\draw (1,0) node [scale=0.3, circle, draw,fill=black]{};
\draw (2,0) node [scale=0.3, circle, draw,fill=black]{};
\draw (1,2) node [scale=0.3, circle, draw,fill=black]{};
\node [below] at (0,0){\small$2$};
\node [left] at (1,1){\small$1$};
\node [below] at (1,0){\small$3$};
\node [below] at (2,0){\small$4$};
\node [left] at (1,2){\small$5$};
\end{tikzpicture}
&1, 2, 6, 24, 114, 600, 3372, 19824,...  & A054872 
\\ 
\hline
\begin{tikzpicture}[scale=0.4]
\draw [line width=1](0,1)--(1,2)--(2,1)--(1,0);
\draw [line width=1](2,1)--(3,0);
\draw (0,1) node [scale=0.3, circle, draw,fill=black]{};
\draw (2,1) node [scale=0.3, circle, draw,fill=black]{};
\draw (1,0) node [scale=0.3, circle, draw,fill=black]{};
\draw (3,0) node [scale=0.3, circle, draw,fill=black]{};
\draw (1,2) node [scale=0.3, circle, draw,fill=black]{};
\node [below] at (0,1){\small$2$};
\node [right] at (2,1){\small$3$};
\node [below] at (1,0){\small$4$};
\node [below] at (3,0){\small$5$};
\node [right] at (1,2){\small$1$};
\end{tikzpicture}
&1, 2, 6, 24, 112, 568, 3032, 16768,...  &  \cellcolor{gray!20!} A118376  
\\ 
\hline
\begin{tikzpicture}[scale=0.4]
\draw [line width=1](0,0)--(1,1)--(2,2);
\draw [line width=1](0,2)--(1,1)--(2,0);
\draw (0,0) node [scale=0.3, circle, draw,fill=black]{};
\draw (0,2) node [scale=0.3, circle, draw,fill=black]{};
\draw (1,1) node [scale=0.3, circle, draw,fill=black]{};
\draw (2,2) node [scale=0.3, circle, draw,fill=black]{};
\draw (2,0) node [scale=0.3, circle, draw,fill=black]{};
\node [below] at (0,0){\small$3$};
\node [above] at (0,2){\small$1$};
\node [below] at (2,0){\small$4$};
\node [below] at (1,1){\small$5$};
\node [above] at (2,2){\small$2$};
\end{tikzpicture}
& 1, 2, 6, 24, 116, 632, 3720, 23072,... &  A212198 
\\ 
\hline
\begin{tikzpicture}[scale=0.4]
\draw [line width=1](0,0)--(0,1)--(1,0)--(1,1)--(0,0);
\draw [line width=1](1,0)--(1,-1);
\draw (0,0) node [scale=0.3, circle, draw,fill=black]{};
\draw (0,1) node [scale=0.3, circle, draw,fill=black]{};
\draw (1,0) node [scale=0.3, circle, draw,fill=black]{};
\draw (1,1) node [scale=0.3, circle, draw,fill=black]{};
\draw (1,-1) node [scale=0.3, circle, draw,fill=black]{};
\node [left] at (0,-0.1){\small$4$};
\node [left] at (0,1.1){\small$1$};
\node [right] at (1,-0.1){\small$5$};
\node [right] at (1,1.1){\small$2$};
\node [right] at (1,-1){\small$3$};
\end{tikzpicture}
&1, 2, 6, 24, 114, 598, 3336, 19402,...  &  \cellcolor{gray!20!} A228907 
\\ 
\hline
\begin{tikzpicture}[scale=0.4]
\draw [line width=1](0,0)--(1,1)--(1,2)--(1,3);
\draw [line width=1](1,1)--(2,0);
\draw (0,0) node [scale=0.3, circle, draw,fill=black]{};
\draw (1,1) node [scale=0.3, circle, draw,fill=black]{};
\draw (1,3) node [scale=0.3, circle, draw,fill=black]{};
\draw (2,0) node [scale=0.3, circle, draw,fill=black]{};
\draw (1,2) node [scale=0.3, circle, draw,fill=black]{};
\node [below] at (0,0){\small$3$};
\node [left] at (1,1){\small$5$};
\node [left] at (1,3){\small$2$};
\node [below] at (2,0){\small$4$};
\node [left] at (1,2){\small$1$};
\end{tikzpicture}
&1, 2, 6, 24, 118, 672, 4256, 29176,...  & A224295 
\\ 
\hline

\end{tabular}
\end{center}
\caption{Conjectured connections for POPs of length 5. For the highlighted OEIS sequences no interpretation in terms of permutation patterns was known until this work. Enumeration for A054872 comes from \cite{BDPP2000}. Also, A212198 comes from \cite{KR12,MS18}. A224295 was obtained using the methods developed in \cite{N13}.}\label{tab-pop-5-conj}
\end{table}

Permutations of length $n$ are called $n$-permutations in this paper, and $S_n$ denotes the set of all $n$-permutations. For an $n$-permutation $\pi$, the {\em complement} $c(\pi)$ of $\pi$ is obtained from $\pi$ by replacing each element $x$ by $n+1-x$. The same operation is well-defined on labels 
$\{1,2,\ldots,n\}$ of an $n$-element poset. Also, the {\em reverse} $r(\pi)$ of $\pi$ is obtained by writing the elements of $\pi$ in the reverse order. The complement, reverse, and usual group theoretical inverse are known as {\em trivial bijections} from $S_n$ to $S_n$.

We let $S_n(p)$ be the set of $n$-permutations avoiding $p$. Patterns $p_1$ and $p_2$ are {\em Wilf-equivalent} if for $n\geq 0$, $|S_n(p_1)|=|S_n(p_2)|$. The definition of Wilf-equivalence can be naturally extended to the simultaneous avoidance of sets of patterns. For example, Remark~\ref{rem-pattern-3} discusses three non-trivially Wilf-equivalent to each other triples of patterns.

Throughout this paper, we let $a(n)$ denote the number of $n$-permutations avoiding 
a pattern $p$ in question, that is, $a(n)=|S_n(p)|$. Occasionally, we introduce other sequences of numbers like 
$b(n)$, but their meaning is defined explicitly in the context.  We also let g.f. stand for ``generating function''. Finally, in this paper, we let $F(n)$ denote the {\em $n$-th Fibonacci number} defined by $F(0)=F(1)=1$ and $F(n)=F(n-1)+F(n-2)$ for $n\geq 2$.

This paper is organized as follows. In Section~\ref{general-sec} we provide a number of general results on certain posets of arbitrary length used by us to explain 12 connections to the OEIS. In Sections~\ref{POP4-sec} and~\ref{POP5-sec} we explain connections for POPs of length 4 and 5, respectively. Finally, in Section~\ref{final-sec} we provide some concluding remarks and state a number of directions of further research.

\section{General results}\label{general-sec}

We begin with the following theorem that allows to obtain results for many more POPs based on already obtained results.

\begin{figure}[htbp]
  \centering
\begin{tikzpicture}[scale=0.8]

\draw [line width=1](0,0)--(2,1.5);
\draw [line width=1](1,0)--(2,1.5);
\draw [line width=1](2,0)--(2,1.5);
\draw [line width=1](4,0)--(2,1.5);

\draw (0,0) node [scale=0.4, circle, draw,fill=black]{};
\draw (1,0) node [scale=0.4, circle, draw,fill=black]{};
\draw (2,0) node [scale=0.4, circle, draw,fill=black]{};
\draw (4,0) node [scale=0.4, circle, draw,fill=black]{};
\draw (2,1.5) node [scale=0.4, circle, draw,fill=black]{};

\node [below] at (0,0){$x_2$};
\node [below] at (1,0){$x_3$};
\node [below] at (2,0){$x_4$};
\node [below] at (4,0){$x_k$};
\node [above] at (2,1.5){$x_1$};

\draw (2.5,0.2) node [scale=0.15, circle, draw,fill=black]{};
\draw (2.75,0.2) node [scale=0.15, circle, draw,fill=black]{};
\draw (3,0.2) node [scale=0.15, circle, draw,fill=black]{};

\end{tikzpicture}
\caption{The POP in Theorem~\ref{thm-B1}.}
 \label{pic-B1}
\end{figure}
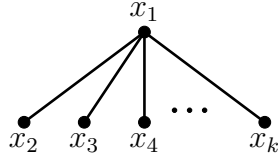

\begin{thm}\label{trivial-sym-thm} Let $p$ be a POP of size $k$. Also, let $p'$ be the POP obtained from $p$ by applying the complement operation on its labels, that is, by replacing a label $x$ by $k+1-x$. Moreover, let  $p''$ be the POP obtained from $p$ by flipping the poset with respect to a horizontal line. Then, $|S_n(p)|=|S_n(p')|=|S_n(p'')|$ for any $n\geq 0$. \end{thm}

\begin{proof} 
Note that an $n$-permutation $\pi$ avoids $p$ if and only if the reverse $r(\pi)$ avoids $p'$. Also, $\pi$ avoids $p$ if and only if the complement $c(\pi)$ avoids $p''$. The reverse and complement operations give trivial bijections from $S_n$ to $S_n$ completing our proof.
\end{proof}

Theorems~\ref{thm-9} and~\ref{thm-A5} below are immediate corollaries of the next theorem.

\begin{thm}\label{thm-B1} Let $p$ be the POP in Figure~\ref{pic-B1}, where 
 $\{x_1,\ldots,x_k\}=\{1,\ldots,k\}$ and $k\geq 1$. Then, 
$$a(n)=\left\{ \begin{array}{ll} 
n! & \mbox{if }n<k\\ 
(k-1)!(k-1)^{n-k+1} & \mbox{if } n\geq k.\end{array}\right.$$
Also,
$$\sum_{n\geq 0}a(n)x^n=\frac{(k-1)(k-1)!x^k}{1-(k-1)x}+\sum_{i=0}^{k-1}i!x^i.$$
\end{thm}

\begin{proof} The base case is straightforward to see, because $p$ cannot occur in such permutations. Now, suppose that $n\geq k$. Note that the element $n$ can only be in one of the 
$x_1-1$ leftmost positions, or in $k-x_1$ rightmost positions, or else $p$ will occur. Thus, we have $k-1$ possibilities to place $n$. Clearly,
for any valid placement of $n$, we have $a(n-1)$ such $n$-permutations. So, $a(n)=(k-1)a(n-1)$ giving the desired recursion since $a(k-1)=(k-1)!$. The g.f. is straightforward to derive.
\end{proof}

Theorems~\ref{thm-10} and~\ref{thm-A12} below are immediate corollaries of the next theorem.

\begin{thm}\label{thm-B2} Let $p$ be the POP in Figure~\ref{pic-B2}, where 
 $k\geq 2$ (if $k=2$ then $p$ is two independent elements). Then, $$a(n)=\left\{ \begin{array}{ll} 
n! & \mbox{if }n<k\\ 
2(k-2) \cdot a(n-1) - (k-2)(k-3)\cdot a(n-2) & \mbox{if } n\geq k.\end{array}\right.$$
Also, 
$$\sum_{n\geq 0}a(n)x^n=\frac{A(x)-B(x)+C(x)}{1-2(k-2)x + (k-2)(k-3)x^2},$$
where $$A(x) = \sum_{i=0}^{k-3}i!x^i,\  B(x)= 2(k-2)\sum_{i=0}^{k-4}i!x^{i+1}, C(x)=(k-2)(k-3)\sum_{i=0}^{k-5}i!x^{i+2}.$$
 \end{thm}

\begin{figure}[htbp]
  \centering
\begin{tikzpicture}[scale=0.8]

\draw [line width=1](0,0)--(1,1.5);
\draw [line width=1](1,0)--(1,1.5);
\draw [line width=1](4,0)--(1,1.5);
\draw [line width=1](0,0)--(3,1.5);
\draw [line width=1](1,0)--(3,1.5);
\draw [line width=1](4,0)--(3,1.5);

\draw (0,0) node [scale=0.4, circle, draw,fill=black]{};
\draw (1,0) node [scale=0.4, circle, draw,fill=black]{};
\draw (4,0) node [scale=0.4, circle, draw,fill=black]{};
\draw (1,1.5) node [scale=0.4, circle, draw,fill=black]{};
\draw (3,1.5) node [scale=0.4, circle, draw,fill=black]{};

\node [below] at (0,-0.1){$2$};
\node [below] at (1,-0.1){$3$};
\node [below] at (4,-0.1){$ k-1$};
\node [above] at (1,1.5){$1$};
\node [above] at (3,1.5){$k$};

\draw (2,0.1) node [scale=0.15, circle, draw,fill=black]{};
\draw (2.25,0.1) node [scale=0.15, circle, draw,fill=black]{};
\draw (2.5,0.1) node [scale=0.15, circle, draw,fill=black]{};

\end{tikzpicture}
\caption{The POP in Theorem~\ref{thm-B2}.}
 \label{pic-B2}
\end{figure}
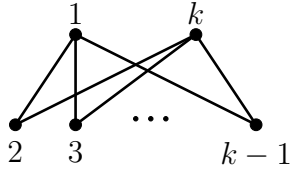

\begin{proof} The base case is easy to see. Let $n\geq k$, $K= \{1,2,\ldots,k-2\}$ and 
$\pi=\pi_1\pi_2\cdots\pi_n$ is a $p$-avoiding $n$-permutation. Note that 
$\{\pi_1,\pi_n\}\cap K \neq \emptyset$  or else the elements in $\{\pi_1,\pi_n\}\cup K$ will form an occurrence of $p$ in $\pi$. If $x\in K$ and $\pi_1=x$  then we have $a(n-1)$ such permutations 
because $x$ cannot be involved in an occurrence of $p$ in $\pi$. The same holds true if $\pi_n=x$. Thus, 
$$a(n)=2(k-2)\cdot a(n-1)-2{k-2\choose 2}a(n-2)$$ 
where we subtract the number of permutations beginning and ending with an element in $K$ because they are counted twice by the term $2(k-2)\cdot a(n-1)$. The g.f. can now be easily derived from the recurrence relation.  
\end{proof}

Theorems~\ref{thm-19},~\ref{thm-22},~\ref{thm-16},~\ref{thm-4},~\ref{thm-A1},~\ref{thm-A4} and ~\ref{thm-A10} below are immediate corollaries of the next theorem.

\begin{thm}\label{thm-B3} Let $p$ be a POP of size $k\geq 1$, and the set of labels of isolated (i.e.\ not comparable to any other elements) nodes include
$$I=\{1,2,\ldots,i\}\cup\{k-s+i+1, k-s+i+2,\ldots,k\}$$ 
for $0\leq i\leq s\leq k$. Also, let $p_1$ be the POP obtained from $P$ by removing the elements corresponding to the labels in $I$. Finally, let $a(n)$ (resp., $b(n)$) be the number of $n$-permutations avoiding $p$ (resp., $p_1$). Then, 
$$a(n)=\left\{ \begin{array}{ll} 
n! & \mbox{if } n<k\\ 
\frac{n!}{(n-s)!}\cdot b(n-s) & \mbox{if } n\geq k.\end{array}\right.$$
\end{thm}

\begin{proof} The base case is obvious, so assume $n\geq k$. Let $A$, $C$ and $B$ denote the parts of an $n$-permutation $\pi$ formed by 
the first $i$ elements, the last $s-i$ elements, and the remaining elements, respectively.
If $p$ occurs in $\pi$ then it induces an occurrence of $p_1$, which cannot begin in $A$ (there are insufficiently many elements to the left of it) or end in $C$  (there are insufficiently many elements to the right of it). Thus, $\pi$ is $p$-avoiding if and only if $B$ is $p_1$-avoiding. Since there are no restrictions on the elements in $A$ and $C$, we can choose them in ${n\choose s}$ ways and then order in $s!$ ways. Independently, we can order the elements of $B$ in $b(n-s)$ ways, which completes our proof.   \end{proof}

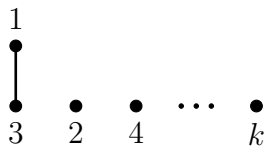
\begin{figure}[htbp]
  \centering
\begin{tikzpicture}[scale=0.8]

\draw [line width=1](0,0)--(0,1);

\draw (0,0) node [scale=0.4, circle, draw,fill=black]{};
\draw (1,0) node [scale=0.4, circle, draw,fill=black]{};
\draw (2,0) node [scale=0.4, circle, draw,fill=black]{};
\draw (0,1) node [scale=0.4, circle, draw,fill=black]{};
\draw (4,0) node [scale=0.4, circle, draw,fill=black]{};

\node [below] at (0,-0.1){$3$};
\node [below] at (1,-0.1){$2$};
\node [below] at (2,-0.1){$4$};
\node [above] at (0,1.1){$1$};
\node [below] at (4,-0.1){$k$};

\draw (2.75,0) node [scale=0.15, circle, draw,fill=black]{};
\draw (3.0,0) node [scale=0.15, circle, draw,fill=black]{};
\draw (3.25,0) node [scale=0.15, circle, draw,fill=black]{};

\end{tikzpicture}
\caption{The POP in Theorem~\ref{thm-B4}.}
 \label{pic-B4}
\end{figure}

Theorem~\ref{thm-22} below is an immediate corollary of the next theorem.

\begin{thm}\label{thm-B4} Let $p$ be the POP in Figure~\ref{pic-B4}, where $k\geq 3$. Then, 
$$a(n)=\left\{ \begin{array}{ll} 
n! & \mbox{if }  n<k\\ 
\frac{n!}{(n-k+3)!}\cdot F(n-k+4) & \mbox{if } n\geq k\end{array}\right.$$
where $F(n)$ is the $n$-th Fibonacci number.\end{thm}

\begin{proof} The case of $n<k$ is trivial, so let $n\geq 3$. The first $n-k+3$ elements of any $p$-avoiding $n$-permutation $\pi$ must avoid the POP $p_1=$ \hspace{-3.5mm}
\begin{minipage}[c]{3.5em}\scalebox{1}{
\begin{tikzpicture}[scale=0.5]

\draw [line width=1](0,-0.5)--(0,0.5);

\draw (0,-0.5) node [scale=0.4, circle, draw,fill=black]{};
\draw (1,-0.5) node [scale=0.4, circle, draw,fill=black]{};
\draw (0,0.5) node [scale=0.4, circle, draw,fill=black]{};

\node [left] at (0,-0.6){${\small 3}$};
\node [right] at (1,-0.6){${\small 2}$};
\node [left] at (0,0.6){${\small 1}$};

\end{tikzpicture}
}\end{minipage}
which is equivalent to avoiding the patterns 231, 312 and 321, simultaneously, and is given by the $(n-k+4)$-th Fibonacci number (see \cite[Table 6.1]{Kit5}). There are no restrictions on the last $k-3$ element of $\pi$ which can be selected in ${n\choose k-3}$ ways and ordered in $(k-3)!$ ways, which complete our proof. Alternatively, we can use Theorem~\ref{thm-B3} for $p$ and $p_1$ in our theorem to obtain the desired result. 
\end{proof}

\begin{figure}[htbp]
  \centering
\begin{tikzpicture}[scale=0.8]

\draw [line width=1](0,0)--(1.5,1.5);
\draw [line width=1](1,0)--(1.5,1.5);
\draw [line width=1](3,0)--(1.5,1.5);

\draw (0,0) node [scale=0.4, circle, draw,fill=black]{};
\draw (1,0) node [scale=0.4, circle, draw,fill=black]{};
\draw (3,0) node [scale=0.4, circle, draw,fill=black]{};
\draw (1.5,1.5) node [scale=0.4, circle, draw,fill=black]{};

\draw (1.75,0) node [scale=0.15, circle, draw,fill=black]{};
\draw (2,0) node [scale=0.15, circle, draw,fill=black]{};
\draw (2.25,0) node [scale=0.15, circle, draw,fill=black]{};

\draw (4,0) node [scale=0.4, circle, draw,fill=black]{};
\draw (5,0) node [scale=0.4, circle, draw,fill=black]{};
\draw (7,0) node [scale=0.4, circle, draw,fill=black]{};

\draw (5.75,0) node [scale=0.15, circle, draw,fill=black]{};
\draw (6.0,0) node [scale=0.15, circle, draw,fill=black]{};
\draw (6.25,0) node [scale=0.15, circle, draw,fill=black]{};

\node [below] at (0,-0.1){$x_2$};
\node [below] at (1,-0.1){$x_3$};
\node [below] at (3,-0.1){$x_{k-s}$};
\node [above] at (1.5,1.5){$x_1$};
\node [below] at (4,-0.1){$y_1$};
\node [below] at (5,-0.1){$y_2$};
\node [below] at (7,-0.1){$y_s$};
\end{tikzpicture}
\caption{The POP in Theorem~\ref{thm-B5}.}
 \label{pic-B5}
\end{figure}
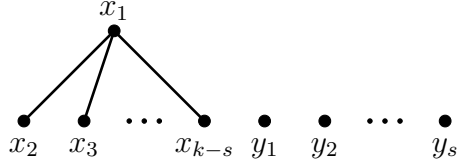

\begin{thm}\label{thm-B5} Let $p$ be the POP in Figure~\ref{pic-B5}, where $k\geq 1$, $0\leq s\leq k$, 
$\{x_1,x_2,\ldots,x_{k-s}\}=\{t,t+1,\ldots,t+k-s-1\}$ for some $t$, $1\leq t\leq s+1$, and $\{y_1,y_2,\ldots,y_k\}=\{1,2,\ldots,t-1\}\cup\{t+k-s,t+k-s+1,\ldots,k\}$. Then, 
$$a(n)=\left\{ \begin{array}{ll} 
n! & \mbox{if } n<k\\ 
\frac{n!(k-s-2)!}{(n-s)!}(k-s-1)^{n-k+s+2} & \mbox{if } n\geq k.\end{array}\right.$$
\end{thm}

\begin{proof} This is an immediate corollary of Theorems~\ref{thm-B1} and~\ref{thm-B3}. \end{proof}

\begin{figure}[htbp]
  \centering
\begin{tikzpicture}[scale=0.8]

\draw [line width=1](0,0)--(0,1);

\draw (0,0) node [scale=0.4, circle, draw,fill=black]{};
\draw (1,0) node [scale=0.4, circle, draw,fill=black]{};
\draw (2,0) node [scale=0.4, circle, draw,fill=black]{};
\draw (0,1) node [scale=0.4, circle, draw,fill=black]{};
\draw (4,0) node [scale=0.4, circle, draw,fill=black]{};

\node [below] at (0,-0.1){$k$};
\node [below] at (1,-0.1){$2$};
\node [below] at (2,-0.1){$3$};
\node [above] at (0,1.1){$1$};
\node [below] at (4,-0.1){$k-1$};

\draw (2.75,0) node [scale=0.15, circle, draw,fill=black]{};
\draw (3.0,0) node [scale=0.15, circle, draw,fill=black]{};
\draw (3.25,0) node [scale=0.15, circle, draw,fill=black]{};

\end{tikzpicture}
\caption{The POP in Theorem~\ref{thm-B6}.}
 \label{pic-B6}
\end{figure}
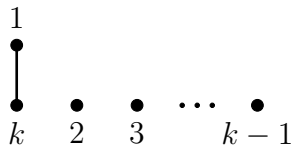

An element $\pi_i$, $1\leq i\leq n$, in a permutation $\pi_1\cdots\pi_n$ is a {\em left-to-right maximum} if $\pi_j<\pi_i$ for $1\leq j<i$.

\begin{thm}\label{thm-B6} Let $p$ be the POP in Figure~\ref{pic-B6}, where $k\geq 3$. Then, $p$-avoiding $n$-permutations are in one-to-one correspondence with $n$-permutations such that for each cycle $c$ the smallest integer interval containing all elements of $c$ has at most $k-1$ elements. 
\end{thm}

\begin{proof} 
We use a standard bijection $\varphi$ from the cyclic structure of a permutation to one-line notation as follows: arrange all cycles in canonical form, where the largest element inside a cycle is written first, and the cycles
are ordered in increasing order of their maximal elements from left-to-right. Erase the parentheses, and observe 
that the maximal elements become the left-to-right maxima in the obtained permutation. For example, $\varphi((163)(7)(82)(45))=54631782$.

Suppose $\pi$ is an $n$-permutation satisfying the condition on its cycles and $x_i$ and $x_j$, $i<j$, are two left-to-right maxima following each other in $\varphi(\pi)=x_1\cdots x_n$. The condition on $\pi$ implies  that in $\varphi(\pi)$  the difference between any two elements in $\{x_i,x_{i+1},\ldots,x_{j-1}\}$ is no more than $k-2$. The same is true for the set $\{x_{\ell},x_{\ell+1},\ldots,x_n\}$, where $x_{\ell}$ is the largest left-to-right maximum. Let $\sigma=\sigma_1\cdots\sigma_n$ be the inverse of $\varphi(\pi)$. We claim that $\sigma$ avoids~$p$. 

Indeed, suppose that $\sigma$ contains a subsequence $\sigma_{i_1}\cdots\sigma_{i_k}$, $1\leq i_1<\cdots<i_k\leq n$, which is an occurrence of $p$.  But then, in $\varphi(\pi)$, we 
have the element $i_k$ to the left of the element $i_1$, and $i_k-i_1>k-2$. If there are no element $z>i_k$ between $i_k$ and $i_1$ in $\varphi(\pi)$, $i_1$ and $i_k$
must be in the same cycle in $\pi$, and we obtain a contradiction. On the other hand, 
if there is  $z>i_k$ between $i_k$ and $i_1$ in $\varphi(\pi)$, consider such a $z$ closest to $i_1$. Then, this $z$ and $i_1$ are in the same cycle in $\pi$, and  $z-i_1>k-1$, which is a contradiction. So, $\sigma$ indeed avoids $p$.

Conversely, consider a cycle in $\pi$ with a maximum element $i_k$ and an element $i_1$ such that $i_k-i_1>k-2$. Then, in $\sigma$, we have the elements in positions $i_1 < i_k$ forming the pattern $21$, and these elements are part of an occurrence of $p$.  This completes our proof.  
\end{proof}

\section{POPs of length 4}\label{POP4-sec}

\subsection{Results coming from our general theorems}

The next theorem gives a new enumerative result on permutations avoiding simultaneously 12 patterns of length 4 in which the first element is larger than the second one.

\begin{thm}\label{thm-19} For the POP $p=$ \hspace{-2.5mm}
\begin{minipage}[c]{2.7em}\scalebox{1}{
\begin{tikzpicture}[scale=0.3]
\draw (0,0) node {};
\draw [line width=1](0,0)--(0,1);
\draw (0,0) node [scale=0.3, circle, draw,fill=black]{};
\draw (0,1) node [scale=0.3, circle, draw,fill=black]{};
\draw (1,0) node [scale=0.3, circle, draw,fill=black]{};
\draw (2,0) node [scale=0.3, circle, draw,fill=black]{};
\node [below] at (0,0){\small$2$};
\node [above] at (0,1){\small$1$};
\node [below] at (1,0){\small$3$};
\node [below] at (2,0){\small$4$};
\end{tikzpicture}  
}\end{minipage}
we have
$$a(n)=\left\{ \begin{array}{ll} 
n! & \mbox{if } n=0,1,2,3\\ 
n(n-1) & \mbox{if } n\geq 4.\end{array}\right.$$
Also, $$\sum_{n\geq 0}a(n)x^n=\frac{1-2x+2x^2+2x^3-x^4}{(1-x)^3}.$$ 
This is the sequence $A103505$ in \cite{oeis}.
\end{thm}

\begin{proof} This is an immediate corollary of Theorem~\ref{thm-B3}, or Theorem~\ref{thm-B5}, with $k=4$, $s=2$ and $b(n)=1$ for $n\geq 0$ because only increasing permutations avoid the pattern $p_1=$  \hspace{-3mm} \begin{minipage}[c]{1.5em}\scalebox{1}{
\begin{tikzpicture}[scale=0.3]
\draw (0,0) node {};
\draw [line width=1](0,0)--(0,1);
\draw (0,0) node [scale=0.3, circle, draw,fill=black]{};
\draw (0,1) node [scale=0.3, circle, draw,fill=black]{};
\node [right] at (0,-0.2){\small$2$};
\node [right] at (0,1.2){\small$1$};
\end{tikzpicture}  
}\end{minipage}
. The g.f. can now be easily derived from the recurrence relation. \end{proof}

The next theorem gives a new enumerative result on permutations avoiding simultaneously 12 patterns of length 4 in which the first element is larger than the third one. 

\begin{thm}\label{thm-22} 
For the POP $p=$ \hspace{-2.5mm}
\begin{minipage}[c]{2.7em}\scalebox{1}{
\begin{tikzpicture}[scale=0.3]
\draw [line width=1](0,0)--(0,1);
\draw (0,0) node [scale=0.3, circle, draw,fill=black]{};
\draw (0,1) node [scale=0.3, circle, draw,fill=black]{};
\draw (1,0) node [scale=0.3, circle, draw,fill=black]{};
\draw (2,0) node [scale=0.3, circle, draw,fill=black]{};
\node [below] at (0,0){\small$3$};
\node [above] at (0,1){\small$1$};
\node [below] at (1,0){\small$2$};
\node [below] at (2,0){\small$4$};
\end{tikzpicture}
}\end{minipage}
we have $$a(n)=\left\{ \begin{array}{ll} 
n! & \mbox{if }  n<4\\ 
n\cdot F(n) & \mbox{if } n\geq 4\end{array}\right.$$
Also, $$\sum_{n\geq 0}a(n)x^n
=\frac{x^4+3 x^3-x^2-x+1}{\left(1-x-x^2\right)^2}.$$
This is essentially the sequence $A045925$ in \cite{oeis} ($a(0)=0$ in $A045925$).
\end{thm}

\begin{proof} This is an immediate corollary of Theorem~\ref{thm-B4}, with $k=4$. The g.f. is not difficult to derive using \end{proof}

The $p$-avoiding $n$-permutations in the next theorem are equinumerous with $S_2(1,n)$ in \cite{BurKit} corresponding to a POP whose poset has the same shape as $p$, but the patterns in question are vincular.

\begin{thm}\label{thm-16} 
For the POP $p=$ \hspace{-3.5mm}
\begin{minipage}[c]{4.8em}\scalebox{1}{
\begin{tikzpicture}[scale=0.3]
\draw [line width=1](0,0)--(1,1)--(2,0);
\draw (0,0) node [scale=0.3, circle, draw,fill=black]{};
\draw (1,1) node [scale=0.3, circle, draw,fill=black]{};
\draw (2,0) node [scale=0.3, circle, draw,fill=black]{};
\draw (3,0) node [scale=0.3, circle, draw,fill=black]{};
\node [left] at (0,-0.2){\small$2$};
\node [right] at (1,1.2){\small$1$};
\node [left] at (2,-0.2){\small$3$};
\node [right] at (3,-0.2){\small$4$};
\end{tikzpicture}
}\end{minipage}
we have $a(0) = a(1) = 1$ and for $n > 1$, $a(n) = n2^{n-2}$.
Also, $$\sum_{n\geq 0}a(n)x^n
=\frac{1-3x+2x^2+2x^3}{\left(1-2x\right)^2}.$$
This is the sequence $A129952$ in \cite{oeis}. Also, this is essentially the sequence $A057711$  in \cite{oeis} ($a(0)=0$ in $A057711$). \end{thm}

\begin{proof} This is an immediate corollary of Theorem~\ref{thm-B3} with $k=4$, $s=1$ and  $p_1=$ \hspace{-3.5mm}
\begin{minipage}[c]{4.8em}\scalebox{1}{
\begin{tikzpicture}[scale=0.3]
\draw [line width=1](0,0)--(1,1)--(2,0);
\draw (0,0) node [scale=0.3, circle, draw,fill=black]{};
\draw (1,1) node [scale=0.3, circle, draw,fill=black]{};
\draw (2,0) node [scale=0.3, circle, draw,fill=black]{};
\node [left] at (0,-0.2){\small$2$};
\node [right] at (1,1.2){\small$1$};
\node [right] at (2,-0.2){\small$3$};
\end{tikzpicture}
}\end{minipage} because $p_1$-avoiding $n$-permutations are precisely $n$-permutations avoiding the patterns $312$ and $321$, simultaneously, and their number is given by $2^{n-1}$ (by applying the complement to the patterns 123 and 132 in \cite[Table 6.1]{Kit5}).  \end{proof}

The next theorem gives a new enumerative result on permutations avoiding simultaneously the patterns in $\{3214, 3124, 4123, 4213, 4321, 4231\}$. 

\begin{thm}\label{thm-9}
For the POP $p=$ \hspace{-2.5mm}
\begin{minipage}[c]{2.8em}\scalebox{1}{
\begin{tikzpicture}[scale=0.3]
\draw [line width=1](0,0)--(1,1)--(2,0);
\draw [line width=1](1,1)--(1,0);
\draw (0,0) node [scale=0.3, circle, draw,fill=black]{};
\draw (1,1) node [scale=0.3, circle, draw,fill=black]{};
\draw (2,0) node [scale=0.3, circle, draw,fill=black]{};
\draw (1,0) node [scale=0.3, circle, draw,fill=black]{};
\node [below] at (0,0){\small$2$};
\node [right] at (1,1.2){\small$1$};
\node [below] at (2,0){\small$4$};
\node [below] at (1,0){\small$3$};
\end{tikzpicture}}\end{minipage}
we have
$$a(n)=\left\{ \begin{array}{ll} 
n! & \mbox{if }n<4\\ 
2\cdot 3^{n-2} & \mbox{if } n\geq 4.\end{array}\right.$$
Also,
$$\sum_{n\geq 0}a(n)x^n=\frac{1-2x-x^2}{1-3 x}.$$
This is the sequence $A025192$ in \cite{oeis}.
\end{thm}

\begin{proof} This is an immediate corollary of Theorem~\ref{thm-B1}.\end{proof}

\begin{thm}\label{thm-10}
For the POP $p=$ \hspace{-3.5mm}
\begin{minipage}[c]{3.5em}\scalebox{1}{
\begin{tikzpicture}[scale=0.3]
\draw [line width=1](0,0)--(0,1)--(1,0)--(1,1)--(0,0);
\draw (0,0) node [scale=0.3, circle, draw,fill=black]{};
\draw (0,1) node [scale=0.3, circle, draw,fill=black]{};
\draw (1,0) node [scale=0.3, circle, draw,fill=black]{};
\draw (1,1) node [scale=0.3, circle, draw,fill=black]{};
\node [left] at (0,-0.1){\small$2$};
\node [left] at (0,1.1){\small$1$};
\node [right] at (1,-0.1){\small$3$};
\node [right] at (1,1.1){\small$4$};
\end{tikzpicture}
}\end{minipage}
we have
$$a(n)=\left\{ \begin{array}{ll} 
n! & \mbox{if }n<4\\ 
4  a(n-1) - 2 a(n-2) & \mbox{if } n\geq 4.\end{array}\right.$$
Also, 
$$\sum_{n\geq 0}a(n)x^n=\frac{1-3x}{1-4x + 2x^2}.$$
This is the sequence $A006012$ in \cite{oeis}.
\end{thm}

\begin{proof} This is an immediate corollary of Theorem~\ref{thm-B2}. \end{proof}

\begin{thm}
\label{thm-4}
For the POP $p=$ \hspace{-3.5mm}
\begin{minipage}[c]{3.5em}\scalebox{1}{
\begin{tikzpicture}[scale=0.3]
\draw [line width=1](0,0)--(0,1)--(0,2);
\draw (0,0) node [scale=0.3, circle, draw,fill=black]{};
\draw (0,1) node [scale=0.3, circle, draw,fill=black]{};
\draw (0,2) node [scale=0.3, circle, draw,fill=black]{};
\draw (1,0) node [scale=0.3, circle, draw,fill=black]{};
\node [left] at (0,-0.1){\small$2$};
\node [left] at (0,1){\small$1$};
\node [left] at (0,2.1){\small$3$};
\node [right] at (1,-0.1){\small$4$};
\end{tikzpicture}

}\end{minipage}
we have $a(n)={2n-2\choose n-1}$ for $n\geq 1$. This is the sequence $A000984$ in \cite{oeis}.\end{thm}

\begin{proof} This is an immediate corollary of Theorem~\ref{thm-B3} because the number of $213$-avoiding $n$-permutations is well-known to be given by the Catalan numbers $\frac{1}{n+1}{2n\choose n}$ (e.g.\  see \cite{Kit5}). Alternatively, note that $p$-avoiding permutations can be obtained from the permutations enumerated in \cite[Cor 5.1]{D18} by applying reverse, then complement, and then inverse.\end{proof}

\subsection{POPs of length 4 with longest chain of size 2} 

The next theorem gives a new enumerative result on permutations avoiding simultaneously 12 patterns of length 4 in which the first element is larger than the last one. 

\begin{thm}\label{thm-24} 
For the POP $p=$ \hspace{-3.5mm}
\begin{minipage}[c]{3.5em}\scalebox{1}{
\begin{tikzpicture}[scale=0.3]
\draw [line width=1](0,0)--(0,1);
\draw (0,0) node [scale=0.3, circle, draw,fill=black]{};
\draw (0,1) node [scale=0.3, circle, draw,fill=black]{};
\draw (1,0) node [scale=0.3, circle, draw,fill=black]{};
\draw (2,0) node [scale=0.3, circle, draw,fill=black]{};
\node [below] at (0,-0.2){\small$4$};
\node [left] at (0,1.2){\small$1$};
\node [below] at (1,-0.2){\small$2$};
\node [below] at (2,-0.2){\small$3$};
\end{tikzpicture}
}\end{minipage}
we have $a(0)=a(1)=1$, $a(2)=2$, $a(3)=6$, and for $n\geq 4$,  $a(n) = a(n-1)+a(n-2)+3a(n-3)+a(n-4)$. Also, 
$$A(x):=\sum_{n\geq 0}a(n)x^n=\frac{1}{1-x-x^2-3x^3-x^4}.$$
 This is the sequence $A214663$ in \cite{oeis}. Also, this is essentially the sequence $ A232164$ in \cite{oeis} ($a(0)=0$  in $ A232164$, and $a(n)$ is $a(n-1)$ in our sequence).
\end{thm}

\begin{proof} We will derive the g.f. $A(x)$ from which it is straightforward to check the recurrence relation for $a(n)$. 
We think of generating all $p$-avoiding $n$-permutations by inserting $n$ in $p$-avoiding $(n-1)$-permutations. Note that $n$ in such a permutation $\pi=\pi_1\pi_2\cdots\pi_n$ can only be in positions $n-2$, $n-1$ or $n$ (otherwise, we would have an occurrence of $p$ involving $n$).Thus, we have three cases to consider. \\[-2mm]

\noindent
{\bf Case 1.} If $\pi_n=n$ then we clearly have $a(n-1)$ such permutations. \\[-2mm]

\noindent
{\bf Case 2.} If $\pi_{n-1}=n$ then $\pi$ may contain $p$ involving $\pi_n$. This happens
when $\pi$ ends with $xynz$ where $x>z$. Indeed, there cannot be any other elements between $x$ and $z$ other than 
$n$ and $y$, because otherwise before inserting $n$, we would have at least two elements between $x$ and $z$, which contradicts us starting with a $p$-avoiding $(n-1)$-permutation.

Letting $b(n)$ be the number of $p$-avoiding $n$-permutations with 
the element in position $n-2$ larger than the element in position $n$, we see that in Case 2 we have $a(n-1)-b(n-1)$ $p$-avoiding $n$-permutations. \\[-2mm]

\noindent
{\bf Case 3.} If  $\pi_{n-2}=n$ then two subcases are possible when inserting $n$ leads to an occurrence of $p$.
\begin{itemize}
\item[(a)] $\pi_n$ is involved in an occurrence of $p$. The number of such permutations is given by $b(n-1)$ because removing $n$ from $\pi$ we have a $p$-avoiding $(n-1)$-permutation with the element in position $n-2$ larger than the element in position $n$. 
\item[(b)] $\pi_n$ is not involved in an occurrence of $p$, but $\pi_{n-1}$ is involved. In this case  $\pi$ must end with $uxnyz$, where $z>y$, $z>u$ (otherwise $\pi_n$ is also involved in an occurrence of $p$ which is impossible in this case), $u>y$, so $uxny$ is the only occurrence of $p$ (otherwise we would have an occurrence of
$p$ before inserting $n$). This implies that $z=n-1$ because otherwise $n-1$ with $z$ would be involved in an occurrence of 
$p$ since $y$ cannot be $n-1$. So, the element $\pi_n=n-1$ can be removed and the rest is counted like in case (a) by $b(n-2)$ （after removing $\pi_n$, we can think of inserting the largest element in a $p$-avoiding $(n-2)$-permutation in the next to last position).
\end{itemize}
\noindent
Summarizing (a) and (b) in Case 3 we have $a(n-1)-b(n-1)-b(n-2)$ $p$-avoiding $n$-permutations.

From Cases 1--3 it follows that for $n\geq 3$ 
\begin{equation}\label{eq-24-1} a(n) = 3a(n-1) - 2b(n-1) -b(n-2) \end{equation}
with initial conditions $a(0)=a(1)=1$, $a(2)=2$ and $b(0)=b(1)=b(2)=0$. Multiplying both parts of \eqref{eq-24-1} by $x^n$ and summing over all $n\geq 3$ we have
\begin{equation}\label{eq-24-2} (1-3x)A(x) + (x^2+2x)B(x) = 1-2x-x^2 \end{equation}
where $B(x)$ is the g.f. for the sequence $b(n)$.

Next, we derive a recurrence relation for $b(n)$. Recall that $b(n)$ counts $p$-avoiding $n$-permutations ending with $xyz$ where $x>z$. Note that no element to the left of $x$ is larger than $z$. We consider three subcases depending on the relative position of $y$ with respect to $x$ and $z$.
\begin{itemize}
\item[(i)] $y>x$, in which case we must have $y=n$, $x=n-1$ and $z=n-2$. Clearly there are $a(n-3)$ such permutations because appending 
the largest element to the right of a $p$-avoiding $(n-3)$-permutation cannot introduce an occurrence of $p$.
\item[(ii)] $x>y>z$, in which case we must have $x=n$, $y=n-1$ and $z=n-2$. Similarly to (i) we have $a(n-3)$ permutations in this case.
\item[(iii)] $y<z$, in which case we must have $x=n$ and$ z=n-1$. $z$ can be removed from $\pi$ and the rest of the counting problem is the same as inserting the largest element in position $n-2$ in a $p$-avoiding $(n-1)$-permutation, 
which is considered in Case 2. Thus, in this case we have $a(n-2)-b(n-2)$ such permutations.
\end{itemize}
\noindent
Summarizing (i)--(iii) we have the following recurrence relation for $n\geq 3$:
\begin{equation}\label{eq-24-3}  b(n) = 2a(n-3) + a(n-2) - b(n-2)  \end{equation}
with initial conditions $a(0)=a(1)=1$, $a(2)=2$ and  $b(0)=b(1)=b(2)=0$. Multiplying both parts of \eqref{eq-24-3} by $x^n$ and summing over all $n\geq 3$ we have
\begin{equation}\label{eq-24-4}(2x^3+x^2)A(x) - （x^2+1）B(x) = x^2  \end{equation}

\noindent
Solving \eqref{eq-24-2} and \eqref{eq-24-4} for $A(x)$ and $B(x)$ we get the desired formula for $A(x)$.
%
\end{proof}

The next theorem gives a new enumerative result on permutations avoiding simultaneously the patterns in $\{2134, 3124, 4123, 3214, 4213, 4312\}$. 

\begin{thm}\label{thm-18}
For the POP $p=$ \hspace{-3.5mm}
\begin{minipage}[c]{3.5em}\scalebox{1}{
\begin{tikzpicture}[scale=0.3]
\draw [line width=1](0,0)--(0,1);
\draw [line width=1](1,0)--(1,1);
\draw (0,0) node [scale=0.3, circle, draw,fill=black]{};
\draw (0,1) node [scale=0.3, circle, draw,fill=black]{};
\draw (1,0) node [scale=0.3, circle, draw,fill=black]{};
\draw (1,1) node [scale=0.3, circle, draw,fill=black]{};
\node [left] at (0,-0.2){\small$2$};
\node [left] at (0,1.2){\small$1$};
\node [right] at (1,-0.2){\small$3$};
\node [right] at (1,1.2){\small$4$};
\end{tikzpicture}
}\end{minipage}
we have $a(0)=1$, and for $n\geq 1$, $a(n) = 2 a(n-1) + 2^{n-1} - 2$, so that
$$a(n)=(n-2)2^{n-1} + 2.$$
Also, $$\sum_{n\geq 0}a(n)x^n = \frac{1-4x+5x^2}{(1-x)(1-2x)^2}.$$
This is the sequence $A048495$ in \cite{oeis}.
\end{thm}

\begin{proof} The initial condition is easy to check, so assume that $n\geq 1$. We next derive the recurrence relation for $a(n)$.

Consider the element 1 in a $p$-avoiding $n$-permutation. If 1 is in the first or last positions, then 
clearly it cannot contribute to an occurrence of $p$, so these cases give $2a(n-1)$ possibilities. 
Now, if 1 is in position $i$, $2\leq i\leq n-1$, then the elements to the left of 1 must be in increasing order to avoid $p$ involving the element 1. Similarly, all elements to the right of 1 must be in decreasing order to avoid $p$ involving the element 1.Thus, the number of 
$p$-avoiding $n$-permutations in this case is given by ${n-1 \choose i-1}$, which is the number of ways to choose the elements to the left of 1. Summing up over all $i$ gives $2^{n-1}-2$ possibilities giving the recurrence relation.  It is straightforward to prove by induction that the desired formula satisfies the recurrence relation. The g.f. is also easy to derive.\end{proof}

The next theorem gives a new enumerative result on permutations avoiding simultaneously the patterns in $\{2314, 3214, 3124, 4213, 4123, 4132\}$. 

\begin{thm}\label{thm-21}
For the POP $p=$ \hspace{-3.5mm}
\begin{minipage}[c]{3.5em}\scalebox{1}{
\begin{tikzpicture}[scale=0.3]
\draw [line width=1](0,0)--(0,1);
\draw [line width=1](1,0)--(1,1);
\draw (0,0) node [scale=0.3, circle, draw,fill=black]{};
\draw (0,1) node [scale=0.3, circle, draw,fill=black]{};
\draw (1,0) node [scale=0.3, circle, draw,fill=black]{};
\draw (1,1) node [scale=0.3, circle, draw,fill=black]{};
\node [left] at (0,-0.2){\small$3$};
\node [left] at (0,1.2){\small$1$};
\node [right] at (1,-0.2){\small$2$};
\node [right] at (1,1.2){\small$4$};
\end{tikzpicture}
}\end{minipage}
we have $a(0)=a(1)=1$, $a(2)=2$, and for $n\geq 3$, $a(n)= 2a(n-1)+2a(n-2)+2a(n-3)$, so that for $n\geq 1$,
\begin{equation}\label{thm-21-formula}
a(n)=\sum_{j=0}^{n-1}\sum_{i=0}^{\lfloor (n-j-1)/2\rfloor} {n-j-i-1\choose i}{j\choose n-j-i-1}2^j.
\end{equation}
Also, $$\sum_{n\geq 0}a(n)x^n=\frac{1-x-2x^2-2x^3}{1-2x-2x^2-2x^3}.$$
This is the sequence $A077835$ in \cite{oeis}.
\end{thm}

\begin{proof} 
The initial conditions are easy to check along with $a(3)=6$ satisfying the recursion, so assume that $n\geq 4$.

Note that the class of $p$-avoiding permutations is closed under the operation of reverse, so that we can count $p$-avoiding $n$-permutations in which 1 is to the right of $n$, and then multiply the result by 2. Let $\pi=\pi_1\cdots \pi_n$ be a $p$-avoiding $n$-permutation. We claim that $\pi_n=1$ or $\pi_{n-1}=1$. Indeed, otherwise $n1\pi_{n-1}\pi_n$ is an occurrence of $p$. There are two cases.\\[-2mm]

\noindent
{\bf Case 1.} If $\pi_n=1$ then it does not affect the rest of $\pi$, so we have $a(n-1)$ possibilities. \\[-2mm]

\noindent
{\bf Case 2.} Let $\pi_{n-1}=1$. If $\pi_n=2$ then the elements 1 and 2 do not affect the rest of the permutation and we have $a(n-2)$ such permutations. Now assume that $\pi_n>2$. If 2 is to the right of $n$ then $n21\pi_n$ is an occurrence of $p$. Thus, 2 is to the left of $n$. 

If $\pi_1\neq 2$ and $\pi_2\neq 2$ then $\pi_1\pi_22n$ is an occurrence of $p$. If $\pi_1=2$ then $\pi_n$ must be 3, as otherwise $231\pi_n$ is an occurrence of $p$. But then, the elements 2, 1 and 3 do not affect the rest of $\pi$, so we have $a(n-3)$ such permutations. 
Finally, $\pi_2\neq 2$ because otherwise $\pi_121\pi_n$ is an occurrence of $p$. 

The g.f. is now easy to derive. Also, \eqref{thm-21-formula} is given by \cite[A077835]{oeis}.
\end{proof}

\begin{remark} 
The combinatorial interpretation in \cite[A077835]{oeis} is ``$a(n)$ is the number of ways two opposing basketball teams could score a combined total of $n$ points (counting one point free throws, two point field goals, and three point field goals) considering the order of the scoring as important.'' Based on the proof of Theorem~\ref{thm-21} we can easily encode these objects by $(n+1)$-permutations avoiding the POP $p=$ \hspace{-3.5mm}
\begin{minipage}[c]{3.5em}\scalebox{1}{
\begin{tikzpicture}[scale=0.3]
\draw [line width=1](0,0)--(0,1);
\draw [line width=1](1,0)--(1,1);
\draw (0,0) node [scale=0.3, circle, draw,fill=black]{};
\draw (0,1) node [scale=0.3, circle, draw,fill=black]{};
\draw (1,0) node [scale=0.3, circle, draw,fill=black]{};
\draw (1,1) node [scale=0.3, circle, draw,fill=black]{};
\node [left] at (0,-0.2){\small$3$};
\node [left] at (0,1.2){\small$1$};
\node [right] at (1,-0.2){\small$2$};
\node [right] at (1,1.2){\small$4$};
\end{tikzpicture}
}\end{minipage} as follows. Let a permutation $\pi=\pi_1\cdots\pi_{n+1}$ be $p$-avoiding and $n\geq 1$. If $\pi_{n+1}=1$ (resp., $\pi_1=1$) then team A (resp., B) scored one at the beginning of the game.   If $\pi_n\pi_{n+1}=12$ (resp., $\pi_1\pi_2=21$) then team A (resp., B) scored two at the beginning of the game.  
If $\pi_n\pi_{n+1}=13$ (resp., $\pi_1\pi_2=31$) then team A (resp., B) scored three at the beginning of the game.  The rest is done by induction. \end{remark}

The next theorem gives a new enumerative result on permutations avoiding simultaneously the patterns in $\{2341, 3241, 3142, 4231, 4132, 4123\}$. 

\begin{thm}\label{thm-23} 
For the POP $p=$ \hspace{-3.5mm}
\begin{minipage}[c]{3.2em}\scalebox{1}{
\begin{tikzpicture}[scale=0.3]
\draw [line width=1](0,0)--(0,1);
\draw [line width=1](1,0)--(1,1);
\draw (0,0) node [scale=0.3, circle, draw,fill=black]{};
\draw (0,1) node [scale=0.3, circle, draw,fill=black]{};
\draw (1,0) node [scale=0.3, circle, draw,fill=black]{};
\draw (1,1) node [scale=0.3, circle, draw,fill=black]{};
\node [left] at (0,-0.2){\small$4$};
\node [left] at (0,1.2){\small$1$};
\node [right] at (1,-0.2){\small$2$};
\node [right] at (1,1.2){\small$3$};
\end{tikzpicture}
}\end{minipage}
we have that $a(0)=a(1)=1$, $a(2)=2$, $a(3)=6$, $a(4)=18$, $a(5)=50$, and for $n\geq 6$,
\begin{equation}\label{rec-23-best}
a(n)=4a(n-1)-5a(n-2)+4a(n-6).
\end{equation}
Also, $$\sum_{n\geq 0}a(n)x^n=\frac{(1-x)^3}{1-4x+5x^2-4x^3}.$$
This is the sequence $A271897$ in \cite{oeis}.
\end{thm}

\begin{proof}
We first prove that $a(n)$ satisfies,  for $n\geq 4$, the recursion
\begin{equation}\label{rec-23}
a(n) = a(n-1) + \sum_{i=1}^{n-2}( i^2 a(n-i-2) + i\cdot  a(n-i-1) ) + (n-1)
\end{equation}
with the initial conditions $a(0)=a(1)=1$, $a(2)=2$ and $a(3)=6$.

The initial conditions are easy to see. To derive \eqref{rec-23} for $n\geq 4$, 
we consider three cases depending on the position of the  element $n$ in a $p$-avoiding permutation $\pi=\pi_1\cdots\pi_n$.  \\[-2mm]

\noindent
{\bf Case 1. } $\pi_1=n$. Then $\pi_2\pi_n\cdots\pi_{n-1}$ must be in decreasing order (otherwise $\pi_1$ and $\pi_{n}$ would be involved in an occurrence of $p$). There are no extra restrictions, so we have $n-1$ possibilities in this case, which is the number of ways to pick $\pi_n$. \\[-2mm]

\noindent
{\bf Case 2. } $\pi_n=n$. In this case $n$ cannot be involved in an occurrence of $p$, so we have $a(n-1)$ possibilities. \\[-2mm]

\noindent
{\bf Case 3. } $\pi_{i+1}=n$, where $1\leq i \leq n-2$. Note that in order to avoid $p$ (keeping in mind that $\pi_i\pi_{i+1}$  can be two middle elements in an occurrence of $p$), all elements to the right of $n$ must be larger than each element, if any, in positions $1, 2, ..., i-1$. Also, once the elements to the right of $n$ are known, there are $n-i-1$ ways to order them, since $\pi_{i+2}\pi_{i+2}\cdots\pi_{n-1}$ must be in decreasing order.  We have two subcases.

\begin{itemize}
\item[(1)] $\pi_i$ is less than any element 
to the right of $n$.  Since the last $n-i$ elements are then the largest in $\pi$, they 
do not affect the rest of $\pi$. So there are $a(i)$ permutations of the smallest elements, and $n-i-1$ ways to arrange the elements to the right of $n$. Thus we have $(n-i-1)a(i)$ possibilities, in this case.

\item[(2)] $\pi_i\in \{i+1, i+2, \ldots, n-1\}$. In this case, $\pi_i$ does not affect whatever is to 
the right of $n$, the last $n-i+1$ elements are the largest elements, and they 
do not affect the rest of $\pi$. Thus, we have $n-i-1$ ways to choose $\pi_i$, then $n-i-1$ ways to order elements to the right of $n$, and  $a(i-1)$ 
ways to pick a permutation of the smallest elements. Thus, we have $(n-i-1)^2a(i-1)$ possibilities in total here. 
\end{itemize}
Summing over all $i$, we have $$\sum_{i=1}^{n-2}( (n-i-1)^2a(i-1)+(n-i-1)a(i))=\sum_{i=1}^{n-2}(i^2a(n-i-2)+i\cdot a(n-i-1))$$ possibilities 
in Case 3. Finally, adding up Cases 1--3, we get  \eqref{rec-23}. 

The g.f. can now be obtained by multiplying both sides of \eqref{rec-23} by $x^n$ and summing over all $n\geq 4$, and using  $\sum_{i\geq 0}ix^i=\frac{x}{(1-x)^2}$ and $\sum_{i\geq 0}i^2x^i=\frac{x(1+x)}{(1-x)^3}$. Finally, our g.f. matches the g.f. in \cite[A271897]{oeis} which gives the known recurrence relation \eqref{rec-23-best} for $a(n)$.
\end{proof}

Note $p$-avoiding $n$-permutations in the next theorem are equinumerous with  $n$-permutations avoiding eight patterns in \cite{AAAHHMv05}. However, our permutations cannot be mapped to the permutations in \cite{AAAHHMv05} via trivial bijection, because in our case the monotone pattern 4321 is forbidden, while no monotone pattern is forbidden in \cite{AAAHHMv05}. 

\begin{thm}\label{thm-15}
For the POP $p=$ \hspace{-3.5mm}
\begin{minipage}[c]{4.8em}\scalebox{1}{
\begin{tikzpicture}[scale=0.3]
\draw [line width=1](0,0)--(1,1)--(2,0);
\draw (0,0) node [scale=0.3, circle, draw,fill=black]{};
\draw (1,1) node [scale=0.3, circle, draw,fill=black]{};
\draw (2,0) node [scale=0.3, circle, draw,fill=black]{};
\draw (3,0) node [scale=0.3, circle, draw,fill=black]{};
\node [left] at (0,-0.2){\small$2$};
\node [right] at (1,1.2){\small$1$};
\node [left] at (2,-0.2){\small$4$};
\node [right] at (3,-0.2){\small$3$};
\end{tikzpicture}
}\end{minipage}
we have that 
$$\sum_{n\geq 0}a(n)x^n=\frac{(1-x)^2}{1-3x+2x^2-2x^3}.$$
This is the sequence $A111281$ in \cite{oeis}.
 \end{thm}

\begin{proof}  
The initial conditions are easy to see. Let $n\geq 4$. There are only three possible places for $n$ in a $p$-avoiding permutation $\pi=\pi_1\cdots \pi_n$.  \\[-2mm]

\noindent
{\bf Case 1.} If $\pi_n=n$  then clearly we have $a(n-1)$ possibilities.\\[-2mm]

\noindent
{\bf Case 2.} If  $\pi_{n-1}=n$ then we have $a(n-1)-b(n-1)$ possibilities, where $b(n)$ counts the number of $p$-avoiding 
$n$-permutations that end with $xyz$ and $x>y$ and $x>z$.\\[-2mm]

\noindent
{\bf Case 3.} If $\pi_{n-2}=n$ then we have $a(n-1)-c(n-1)$ possibilities, where $c(n)$ counts the number of $p$-avoiding $n$-permutations that end with $xyzx^+$, where $x>y$, $x>z$ and $x^+$ denotes an element larger than $x$ (note that we are forced to have $\pi_n>x$ to avoid $p$).

Summarizing Cases 1--3, we obtain $a(n) = 3a(n-1) - b(n-1) - c(n-1)$. Multiplying both sides of this relation  by $x^n$, then summing over all $n\geq 4$ and keeping in mind that $a(0)=a(1)=1$, $a(2)=2$, $a(3)=6$, 
$b(0)=b(1)=b(2)=c(0)=c(1)=c(2)=c(3)=0$ and $b(3)=2$ (only 312 and 321 are such permutations), we obtain
\begin{equation}\label{eq-15-1}
(1-3x)A(x) + xB(x) + xC(x) = 1-2x-x^2 
\end{equation}
where $A(x)$, $B(x)$, $C(x)$ are the g.f.s for $a(n)$, $b(n)$, $c(n)$, respectively. 

Next we derive a recurrence relation for $b(n)$. Note that 
in $n$-permutations counted by $b(n)$, $\pi_{n-1}\neq n$ and $\pi_{n}\neq n$. Thus, we must have $x=n$ ($x$ is used in the definition of $b(n)$), and this case is the same as Case 3
above, so we have $b(n)=a(n-1)-c(n-1)$. Multiplying both sides of this relation by $x^n$ and summing over all $n\geq 4$, we obtain
\begin{equation}\label{eq-15-2}
xA(x) - B(x) - xC(x) = x + x^2 
\end{equation}
Finally, we derive a recurrence relation for $c(n)$. Note that we must have $x^+=n$ and $x =n-1$. We can remove $n$ (it cannot contribute an occurrence of $p$), and 
the rest is exactly Case 3 above with the indices shifted by 1. That is, $c(n) = a(n-2) - c(n-2)$.
Multiplying both sides of this relation by $x^n$ and summing over all $n\geq 4$, we obtain
\begin{equation}\label{eq-15-3}
x^2A(x) - （1+x^2）C(x) = x^2 + x^3 
\end{equation}
Solving the system of linear equations \eqref{eq-15-1},  \eqref{eq-15-2} and  \eqref{eq-15-3}, we obtain $A(x)$. 
\end{proof}

\begin{thm}\label{thm-20}
For the POP $p=$ \hspace{-3.5mm}
\begin{minipage}[c]{4.8em}\scalebox{1}{
\begin{tikzpicture}[scale=0.3]
\draw [line width=1](0,0)--(1,1)--(2,0);
\draw (0,0) node [scale=0.3, circle, draw,fill=black]{};
\draw (1,1) node [scale=0.3, circle, draw,fill=black]{};
\draw (2,0) node [scale=0.3, circle, draw,fill=black]{};
\draw (3,0) node [scale=0.3, circle, draw,fill=black]{};
\node [left] at (0,-0.2){\small$3$};
\node [right] at (1,1.2){\small$1$};
\node [left] at (2,-0.2){\small$4$};
\node [right] at (3,-0.2){\small$2$};
\end{tikzpicture}
}\end{minipage}
we have $a(0)=a(1)=1$, and for $n\geq 2$, $a(n)=2(a(n-1)+a(n-2))$. Also, 
$$\sum_{n\geq 0}a(n)x^n = \frac{1-x-2x^2}{1 - 2x - 2x^2}.$$
This is essentially the sequence $A002605$ in \cite{oeis} ($a(0)=0$ there).
\end{thm}

\begin{proof} We think of generating all $p$-avoiding $n$-permutations from $p$-avoiding $(n-1)$-permutations by inserting the element $n$. $n$ can only be in positions $n-2$, $n-1$ and $n$. \\[-2mm]

\noindent
{\bf Case 1.} $n$ is in position $n$. We clearly have $a(n-1)$ such permutations. \\[-2mm]

\noindent
{\bf Case 2.} $n$ is in position $n-1$. We clearly have $a(n-1)$ such permutations. \\[-2mm]

\noindent
{\bf Case 3.} $n$ is in position $n-2$. We clearly have $a(n-1)-b(n-1)$ such permutations, where $b(n-1)$ counts $p$-avoiding $(n-1)$-permutations ending with $xyz$, where $x$ is larger than $y$ and $z$. We need to subtract such permutations because when inserting $n$
in position $n-2$ in them we will get the forbidden pattern $p$. We do not need to subtract any other permutations because if $p$ occurs as $xnyz$, and $x$ is not next to $n$, then we would have an occurrence of $p$ before inserting $n$.  

Summarizing Cases 1--3, we have 
\begin{equation}\label{eq-20-1}
a(n)=3a(n-1)-b(n-1). 
\end{equation}
We next derive a recurrence relation for $b(n)$.  Note that we must have $x=n$ ($x$ appears in the definition of $b(n)$), or else $nxyz$ is an occurrence of $p$, which is impossible. But then, we are back to Case 3 above, since all such permutations can be obtained by inserting the largest element $n$ in position 
$n-2$ in a $p$-avoiding $(n-1)$-permutation. So, we have
\begin{equation}\label{eq-20-2}
b(n)=a(n-1)-b(n-1). 
\end{equation}
Subtracting \eqref{eq-20-2} from \eqref{eq-20-1}, we obtain
\begin{equation}\label{eq-20-3}
b(n) = a(n) - 2a(n-1).
\end{equation}
Using \eqref{eq-20-3} in \eqref{eq-20-1} we get the desired recurrence. The g.f. is now straightforward to derive. \end{proof}

Note that $p$-avoiding $n$-permutations in the next theorem are equinumerous with  $n$-permutations avoiding simultaneously the patterns in $\{1432, 2431,$ $3412, 3421,4132,4231, 4312, 4321\}$
appearing in \cite{AAAHHMv05}. However, our permutations cannot be mapped to the permutations in \cite{AAAHHMv05} via trivial bijection, because in our case no monotone pattern is forbidden, while the pattern 4321 is forbidden in \cite{AAAHHMv05}. 

\begin{thm}\label{thm-17} 
For the POP $p=$ \hspace{-3.5mm}
\begin{minipage}[c]{4.8em}\scalebox{1}{
\begin{tikzpicture}[scale=0.3]
\draw [line width=1](0,0)--(1,1)--(2,0);
\draw (0,0) node [scale=0.3, circle, draw,fill=black]{};
\draw (1,1) node [scale=0.3, circle, draw,fill=black]{};
\draw (2,0) node [scale=0.3, circle, draw,fill=black]{};
\draw (3,0) node [scale=0.3, circle, draw,fill=black]{};
\node [left] at (0,-0.2){\small$1$};
\node [right] at (1,1.2){\small$2$};
\node [left] at (2,-0.2){\small$4$};
\node [right] at (3,-0.2){\small$3$};
\end{tikzpicture}
}\end{minipage}
we have $a(0)=a(1)=1$, and for $n\geq 2$, $a(n) = 3a(n-1) - a(n-2)$. Also,
$$\sum_{n\geq 0}a(n)x^n=\frac{1 - 2x + x^3}{1 - 3x + x^2}.$$
This is the sequence $A111282$ in \cite{oeis}.
\end{thm}

\begin{proof} One can easily check that $a(n)=n!$ for $1\leq n\leq 3$ and the recursion is satisfied, so assume that $n\geq 4$.  We consider possible places for the element $n$ in a $p$-avoiding permutation $\pi=\pi_1\cdots\pi_n$. There are only three possibilities.\\[-2mm]

\noindent
{\bf Case 1.} $\pi_1=n$. In this case we clearly have $a(n-1)$ possibilities.\\[-2mm]

\noindent
{\bf Case 2.}  $\pi_n=n$. In this case we clearly have $a(n-1)$ possibilities.\\[-2mm]

\noindent
{\bf Case 3.}  $\pi_{n-1}=n$. In this case, we have $a(n-1) - b(n-1)$ possibilities where $b(n)$ counts the number of $p$-avoiding 
$n$-permutations that end with $xyz$, where $x<y$ and $z<y$, because in this case inserting $n$ in position $n-1$ will create an occurrence of $p$. 

Summarizing Cases 1--3, we have 
\begin{equation}\label{eq-17-1}
a(n) = 3a(n-1) - b(n-1).
\end{equation}
We next derive a recursion for $b(n)$. Because we deal with $p$-avoiding permutations, there are three possibilities for a position of the element $n$, as mentioned above. 
\begin{itemize}
\item[(i)] If a permutation counted by $b(n)$ begins with $n$, then clearly we have $b(n-1)$ possibilities.
\item[(ii)] No permutation ending with $n$ can be counted by $b(n)$.
\item[(iii)] Finally, any permutation with $n$ in position $n-1$ will be counted by $b(n)$ since $n\geq 4$, but the number of 
these permutations is given by Case 3 above, and it is $a(n-1) - b(n-1)$.
\end{itemize}
Summarizing (i)--(iii), we have $b(n) = a(n-1)$, which together with \eqref{eq-17-1} gives the desired result.  
\end{proof}

Note that, according to \cite[A111277]{oeis}, $p$-avoiding $n$-permutations in the next theorem are equinumerous with  $n$-permutations avoiding simultaneously the patterns in $\{2413, 4213, 2431, 4231, 4321\}$ and with $n$-permutations avoiding simultaneously the patterns in $\{3142, 3412, 3421, 4312, 4321\}$. However, our permutations cannot be mapped to any of these permutations via trivial bijection, because in our case no monotone pattern is forbidden, while the pattern 4321 is forbidden in both of the other cases.

\begin{thm}\label{thm-12} 
For the POP $p=$ \hspace{-3.5mm}
\begin{minipage}[c]{3.8em}\scalebox{1}{
\begin{tikzpicture}[scale=0.3]
\draw [line width=1](0,0)--(0,1)--(1,0)--(1,1);
\draw (0,0) node [scale=0.3, circle, draw,fill=black]{};
\draw (0,1) node [scale=0.3, circle, draw,fill=black]{};
\draw (1,0) node [scale=0.3, circle, draw,fill=black]{};
\draw (1,1) node [scale=0.3, circle, draw,fill=black]{};
\node [left] at (0,-0.1){\small$2$};
\node [left] at (0,1.1){\small$1$};
\node [right] at (1,-0.1){\small$3$};
\node [right] at (1,1.1){\small$4$};
\end{tikzpicture}
}\end{minipage}
we have $a(0)=a(1)=1$ and, for $n\geq 2$, $a(n) = 4a(n-1) - 3a(n-2) + 1$, so that 
\begin{equation}\label{thm-12-formula}
a(n)=\frac{3^n-2n+3}{4}.
\end{equation}
Also, $$\sum_{n\geq 0}a(n)x^n=\frac{(1-2x)^2}{(1-3x)(1-x)^2}.$$
This is the sequence $A111277$ in \cite{oeis}. 
\end{thm}

\begin{proof} 
The initial values are easy to see, so let $n\geq 2$. Clearly, to avoid $p$ is the same as to avoid simultaneously the patterns in $\{4312\}\cup A$, where $A=\{4213, 3214, 4123, 3124\}$. We claim that 
\begin{equation}\label{thm-12-rec-rel}
a(n) = 4a(n-1) - 2a(n-2) - (a(n-2)-1).
\end{equation}
Observe that to avoid the patterns in $A$, an $n$-permutation must begin or end with 1 or 2. This explains the term $4a(n-1)$ in \eqref{thm-12-rec-rel} coming from 
generating all such $n$-permutations from $p$-avoiding $(n-1)$-permutations. However, the $n$-permutations beginning {\em and} ending with 1 and 2 are counted twice, which explains the term $-2a(n-2)$  in \eqref{thm-12-rec-rel}. Next, observe that some of $n$-permutations ending with 2 and counted by $4a(n-1)$ contain occurrence(s) of the pattern 
$4312$, and they need to be subtracted (none of these permutations was subtracted by $-2a(n-2)$ because the element 1 cannot be leftmost). Call these permutations ``bad'' permutations, and denote an occurrence of 4312 by $xy12$ (because the elements 1 and 2 must be involved 
in any occurrence of 4312). 

To see that the number of bad permutations is 
$a(n-2)-1$, and thus to complete the proof of  \eqref{thm-12-rec-rel}, note that the element 1 must be next to the (rightmost) element 2 in any bad permutation. Indeed, if an 
element $z$ is between 1 and 2, then 
\begin{itemize} 
\item $xy1z$ is an occurrence of the pattern 4312 if $z<y$, and 
\item $xy1z$ is an occurrence of the pattern $4213$ or $3214$ 
if $z>y$, 
\end{itemize}
which is impossible, because the $(n-1)$-permutation obtained by removing the element 2 must avoid all 5 
forbidden patterns. But any $(n-2)$-permutation, except for the increasing permutation, that avoid the 5 
patterns with 12 appended to the right will result in a bad permutation, so the number of bad $n$-permutations is indeed $a(n-2) - 1$, and  \eqref{thm-12-rec-rel} is proved. 

Note that the formula \eqref{thm-12-formula} given in \cite[A111277]{oeis} satisfies our recurrence relation. Also, the g.f. is straightforward to derive. \end{proof}

The next theorem gives a new enumerative result on permutations avoiding simultaneously the patterns in $\{4132, 4213, 3214, 4123, 3124\}$. 

\begin{thm}\label{thm-11} 
For the POP $p=$ \hspace{-3.5mm}
\begin{minipage}[c]{3.8em}\scalebox{1}{
\begin{tikzpicture}[scale=0.3]
\draw [line width=1](0,0)--(0,1)--(1,0)--(1,1);
\draw (0,0) node [scale=0.3, circle, draw,fill=black]{};
\draw (0,1) node [scale=0.3, circle, draw,fill=black]{};
\draw (1,0) node [scale=0.3, circle, draw,fill=black]{};
\draw (1,1) node [scale=0.3, circle, draw,fill=black]{};
\node [left] at (0,-0.1){\small$3$};
\node [left] at (0,1.1){\small$1$};
\node [right] at (1,-0.1){\small$2$};
\node [right] at (1,1.1){\small$4$};
\end{tikzpicture}
}\end{minipage}
we have $a(0)=a(1)=1$ and, for $n\geq 2$, $a(n) = 4a(n-1) - 3a(n-2) + a(n-3)$, so that, for $n\geq 1$, 
\begin{equation}\label{thm-11-formula}
a(n)=\sum_{i=0}^{n-1} {n+2i-1\choose 3i}.
\end{equation}
Also, $$\sum_{n\geq 0}a(n)x^n= \frac{1-3x+x^2}{1-4x+3x^2-x^3}.$$
This is the sequence $A052544$ in \cite{oeis}. 
\end{thm}

\begin{proof} 
The initial conditions are easy to check, so assume that $n\geq 2$. Clearly, to avoid $p$ is the same as to avoid simultaneously the patterns in $\{4132\}\cup A$, where $A=\{4213, 3214,4123, 3124\}$.  We begin with proving that 
\begin{equation}\label{thm-11-rec-rel}
a(n) = 4a(n-1) - 2a(n-2) - (a(n-1) - (a(n-2)+\sum_{i=0}^{n-3}a(n-2-i))) 
\end{equation}
Observe that to avoid the patterns in $A$, an $n$-permutation must begin or end with 1 or 2. This explains the term $4a(n-1)$ in \eqref{thm-11-rec-rel} coming from 
generating all such $n$-permutations from $p$-avoiding $(n-1)$-permutations. However, the $n$-permutations beginning and ending with the elements in $\{1, 2\}$ are counted twice, which explains the term $-2a(n-2)$  in \eqref{thm-11-rec-rel}.
Next, observe that some of $n$-permutations ending with 2 and counted by $4a(n-1)$ contain occurrence(s) of the pattern 
$4132$, and they need to be subtracted (none of these permutations was subtracted by $-2a(n-2)$ because the element 1 cannot be leftmost). Call these permutations ``bad'' permutation.

To count bad permutations, we count $n$-permutations ending with 2 which cannot be bad, namely, which avoid all 5 patterns, and then subtract them from all $n$-permutations in question ending with 2.  All such non-bad $n$-permutations are given by considering two cases:
\begin{itemize}
\item[(i)] there are no elements between 1 and 2, that is an $n$-permutation ends with 12. We have $a(n-2)$ such permutations.
\item[(ii)] there is at least one element between 1 and 2, and every element to the left of 1 is less than any element to 
the right of 1 (except for the element 2). In this case, the elements to the left of 1 must be in increasing order, 
or else, we will get an occurrence of the pattern 3214 involving two elements to the left of 1, 1 itself, and an 
element to the right of 1. Between the elements 1 and 2 we can have any permutation avoiding the 5 patterns. Thus, we have $\sum_{i=0}^{n-3}a(n-2-i)$ such permutations. 
\end{itemize}
So, the number of bad permutations is 
$a(n-1) - (a(n-2)+\sum_{i=0}^{n-3}a(n-2-i))$, which completes the proof of \eqref{thm-11-rec-rel}.  

From \eqref{thm-11-rec-rel} we have $a(n)=3a(n-1)-a(n-2) + \sum_{i=0}^{n-3}a(n-2-i)$.
Taking $a(n) - a(n-1)$ leads to the desired recursion
$a(n) = 4a(n-1) - 3a(n-2) + a(n-3)$ that is presented in \cite[A052544]{oeis}. The g.f. can now be easily derived, and the formula \eqref{thm-11-formula} presented in  \cite[A052544]{oeis} can also be checked by Mathematica to satisfy the recursion. 
\end{proof}

\subsection{POPs of length 4 with longest chain of size 3} 
 
\begin{thm}\label{thm-14} 
For the POP $p=$ \hspace{-3.5mm}
\begin{minipage}[c]{3em}\scalebox{1}{
\begin{tikzpicture}[scale=0.3]
\draw [line width=1](0,0)--(0,1)--(0,2);
\draw (0,0) node [scale=0.3, circle, draw,fill=black]{};
\draw (0,1) node [scale=0.3, circle, draw,fill=black]{};
\draw (0,2) node [scale=0.3, circle, draw,fill=black]{};
\draw (1,0) node [scale=0.3, circle, draw,fill=black]{};
\node [left] at (0,-0.1){\small$2$};
\node [left] at (0,1){\small$1$};
\node [left] at (0,2.1){\small$4$};
\node [right] at (1,-0.1){\small$3$};
\end{tikzpicture}
}\end{minipage}
we have, for $n\geq 1$,
$$\sum_{k=0}^{n-1}\frac{1}{n+1}{n-k-1\choose k}{2n-2k\choose n}.$$
This is the sequence $A049124$ in \cite{oeis}.
\end{thm}

\begin{proof} To avoid $p$ is the same as avoiding the patterns 3214, 3124, 2134, and  2143, simultaneously. Applying inverse and then complement  to the these permutations, we obtain permutations enumerated in \cite[Thm 6.2]{D18}. \end{proof}

\begin{thm}\label{thm-8} 
For the POP $p=$ \hspace{-3.5mm}
\begin{minipage}[c]{4em}\scalebox{1}{
\begin{tikzpicture}[scale=0.3]
\draw [line width=1]
(0,0)--(0,1)--(1,2)--(2,1);
\draw (0,0) node [scale=0.3, circle, draw,fill=black]{};
\draw (0,1) node [scale=0.3, circle, draw,fill=black]{};
\draw (1,2) node [scale=0.3, circle, draw,fill=black]{};
\draw (2,1) node [scale=0.3, circle, draw,fill=black]{};
\node [left] at (0,-0.1){\small$4$};
\node [left] at (0,1){\small $2$};
\node [right] at (1,2.1){\small $1$};
\node [right] at (2,1){\small $3$};
\end{tikzpicture}
}\end{minipage}
we have for $a(n)$ the asymptotic growth of $$\frac{(3-\sqrt{5})(7+3\sqrt{5}+3\sqrt{22+10\sqrt{5}})}{4}.$$
The g.f. $A(x)=\sum_{n\geq 0}a(n)x^n$ satisfies 
$$(2x^2+8x-1)A^4(x) + (x^3+4x^2-46x+5)A^3(x) + (3x^3-21x^2+94x-9)A^2(x) +$$ 
$$(x^3+12x^2-82x+7)A(x) + 3x^2+26x-2 = 0.$$
This is the sequence $A257561$ in \cite{oeis}.
\end{thm}

\begin{proof} To avoid $p$ is the same as to avoid the patterns 4231, 4312, and 4321, and these are precisely the patterns in \cite[Thm 3.1]{AHPSV18} and \cite[A257561]{oeis}, from which the results follow. \end{proof}

\begin{thm}\label{thm-5} 
For the POP $p=$ \hspace{-3.5mm}
\begin{minipage}[c]{4em}\scalebox{1}{
\begin{tikzpicture}[scale=0.3]
\draw [line width=1]
(0,0)--(0,1)--(1,2)--(2,1);
\draw (0,0) node [scale=0.3, circle, draw,fill=black]{};
\draw (0,1) node [scale=0.3, circle, draw,fill=black]{};
\draw (1,2) node [scale=0.3, circle, draw,fill=black]{};
\draw (2,1) node [scale=0.3, circle, draw,fill=black]{};
\node [left] at (0,-0.1){\small$2$};
\node [left] at (0,1){\small $3$};
\node [right] at (1,2.1){\small $1$};
\node [right] at (2,1){\small $4$};
\end{tikzpicture}
}\end{minipage}
we have $$\sum_{n\geq 0}a(n)x^n=\frac{1-5x+(1+x)\sqrt{1-4x}}{1-5x+(1-x)\sqrt{1-4x}}.$$
This is the sequence $A111279$ in \cite{oeis}.
\end{thm}

\begin{proof}
Clearly, avoiding $p$ is equivalent to avoiding the patterns 4132, 4231, and 4123, simultaneously. Applying the inverse and reverse operations to these patterns gives the class $\Pi_3$  in  \cite{CM18}, which was enumerated there.\end{proof}

\begin{thm}\label{thm-6} 
For the POP $p=$ \hspace{-3.5mm}
\begin{minipage}[c]{4em}\scalebox{1}{
\begin{tikzpicture}[scale=0.3]
\draw [line width=1]
(0,0)--(0,1)--(1,2)--(2,1);
\draw (0,0) node [scale=0.3, circle, draw,fill=black]{};
\draw (0,1) node [scale=0.3, circle, draw,fill=black]{};
\draw (1,2) node [scale=0.3, circle, draw,fill=black]{};
\draw (2,1) node [scale=0.3, circle, draw,fill=black]{};
\node [left] at (0,-0.1){\small$2$};
\node [left] at (0,1){\small $4$};
\node [right] at (1,2.1){\small $1$};
\node [right] at (2,1){\small $3$};
\end{tikzpicture}
}\end{minipage}
we have, for $n\geq 1$, $$a(n) = \frac{1}{n}\sum_{k=0}^{n-1} {2n-2k-2\choose n-k-1}{n+k-1\choose n-1}.$$
Also, $A(x)=\sum_{n\geq 0}a(n)x^n$ satisfies $$A(x)=1+\frac{xA(x)}{1-xA^2(x)}.$$
This is the sequence $A106228$ in \cite{oeis}.
\end{thm}

\begin{proof} Clearly, avoiding $p$ is equivalent to avoiding the patterns 4213, 4123, and 4132, simultaneously, and these are precisely the patterns  in  \cite[A106228]{oeis}. The formula for $a(n)$, and the relation for $A(x)$, can be found in the OEIS. Note that the relation is coming from \cite{CM18-2} as the triple 242 there is the reverse and complement of our patterns.\end{proof}

\begin{thm}\label{thm-3} 
For the POP $p=$ \hspace{-3.5mm}
\begin{minipage}[c]{4em}\scalebox{1}{
\begin{tikzpicture}[scale=0.3]
\draw [line width=1]
(0,0)--(0,1)--(1,2)--(2,1);
\draw (0,0) node [scale=0.3, circle, draw,fill=black]{};
\draw (0,1) node [scale=0.3, circle, draw,fill=black]{};
\draw (1,2) node [scale=0.3, circle, draw,fill=black]{};
\draw (2,1) node [scale=0.3, circle, draw,fill=black]{};
\node [left] at (0,-0.1){\small$2$};
\node [left] at (0,1){\small $1$};
\node [right] at (1,2.1){\small $3$};
\node [right] at (2,1){\small $4$};
\end{tikzpicture}
}\end{minipage}
we have, $a(0)=a(1)=1$, $a(2)=2$, and for $n\geq 3$,
$$a(n) = \frac{(13n-5)a(n-1) - (16n-23)a(n-2) + 5(n-2)a(n-3)}{2(n+1)}.$$ 
Also,
$$\sum_{n\geq 0}a(n)x^n = \frac{2}{1+x+\sqrt{(1-x)(1-5x)}}.$$
This is the sequence $A033321$ in \cite{oeis}.
\end{thm}

\begin{proof} Clearly, avoiding $p$ is equivalent to avoiding the patterns 2143, 3142 and 3241, simultaneously, which is reverse complement of the patterns considered in \cite{BB16}, where the g.f. is given. The recurrence relation for $a(n)$ is given in \cite[A033321]{oeis}. \end{proof}

\begin{remark}\label{rem-pattern-3} Note that, by \cite[A033321]{oeis}, the avoidance of the pattern \hspace{-3.5mm}
\begin{minipage}[c]{4em}\scalebox{1}{
\begin{tikzpicture}[scale=0.3]
\draw [line width=1]
(0,0)--(0,1)--(1,2)--(2,1);
\draw (0,0) node [scale=0.3, circle, draw,fill=black]{};
\draw (0,1) node [scale=0.3, circle, draw,fill=black]{};
\draw (1,2) node [scale=0.3, circle, draw,fill=black]{};
\draw (2,1) node [scale=0.3, circle, draw,fill=black]{};
\node [left] at (0,-0.1){\small$2$};
\node [left] at (0,1){\small $1$};
\node [right] at (1,2.1){\small $3$};
\node [right] at (2,1){\small $4$};
\end{tikzpicture}
}\end{minipage} in Theorem~\ref{thm-3} is equivalent to the avoidance of the following three triples of patterns not trivially equivalent to each other $\{2431,4231,4321\}$, $\{2413, 3142, 2143\}$, and $\{2143, 3142, 4132\}$.\end{remark}

The $n$-th {\em large Schr\"{o}der number} is defined by the following recurrence relation:
$$\mathcal{S}_n=\mathcal{S}_{n-1}+\sum_{i=0}^{n-1}\mathcal{S}_i\mathcal{S}_{n-1-i}$$
where $\mathcal{S}_0=1$.
\begin{thm}\label{thm-1} 
For the POP $p=$ \hspace{-3.5mm}
\begin{minipage}[c]{4em}\scalebox{1}{
\begin{tikzpicture}[scale=0.3]
\draw [line width=1](0,2)--(1,1)--(1,0);
\draw [line width=1](1,1)--(2,2);
\draw (0,2) node [scale=0.3, circle, draw,fill=black]{};
\draw (1,1) node [scale=0.3, circle, draw,fill=black]{};
\draw (1,0) node [scale=0.3, circle, draw,fill=black]{};
\draw (2,2) node [scale=0.3, circle, draw,fill=black]{};
\node [left] at (0,2){{\small $3$}};
\node [right] at (1,0.9){{\small $1$}};
\node [right] at (1,-0.1){{\small $2$}};
\node [right] at (2,2){{\small $4$}};
\end{tikzpicture}
}\end{minipage}
we have that $a(0)=0$, and for $n\geq 1$, $a(n)=\mathcal{S}_n$, the $(n-1)$-th large  Schr\"{o}der number, so that
$$\sum_{n\geq 0}a(n)x^n = \frac{3-x-\sqrt{1-6x+x^2}}{2}.$$
This is the sequence $A006318$ in \cite{oeis}.
\end{thm}

\begin{proof} Avoding $p$ is equivalent to avoiding 2134 and 2143. The result follows from \cite{Kr03} by applying the reverse and complement to these patterns. \end{proof}

The enumeration for the sequence in the next theorem is unknown, but its g.f. is conjectured to be non-D-finite \cite{AHPSV18}.

\begin{thm}\label{thm-7} 
For the POP $p=$ \hspace{-3.5mm}
\begin{minipage}[c]{4em}\scalebox{1}{
 \begin{tikzpicture}[scale=0.3]
\draw [line width=1]
(0,1)--(1,2)--(2,1)--(1,0)--(0,1);
\draw (0,1) node [scale=0.3, circle, draw,fill=black]{};
\draw (1,2) node [scale=0.3, circle, draw,fill=black]{};
\draw (2,1) node [scale=0.3, circle, draw,fill=black]{};
\draw (1,0) node [scale=0.3, circle, draw,fill=black]{};
\node [left] at (0,1){{\small $2$}};
\node [right] at (1.1,2.1){{\small $1$}};
\node [right] at (2,1){{\small $3$}};
\node [right] at (1.1,-0.2){{\small $4$}};
\end{tikzpicture}
}\end{minipage}
we have that $a(n)$ is given by the sequence $A053617$ in \cite{oeis}.
\end{thm}

\begin{proof} Clearly, to avoid $p$ is the same as to avoid the patterns 4231 and 4321, which is the reverse of the patterns in \cite[A053617]{oeis}. \end{proof}

The enumeration for the sequence in the next theorem is unknown. The sequence appears in \cite{KS03} as avoidance of the patterns 2143 and 2413, and the patterns corresponding to our $p$ are obtained from these by applying inverse. 

\begin{thm}\label{thm-2} 
For the POP $p=$ \hspace{-3.5mm}
\begin{minipage}[c]{4em}\scalebox{1}{
\begin{tikzpicture}[scale=0.3]
\draw [line width=1](0,1)--(1,2)--(2,1)--(1,0)--(0,1);
\draw (0,1) node [scale=0.3, circle, draw,fill=black]{};
\draw (1,2) node [scale=0.3, circle, draw,fill=black]{};
\draw (2,1) node [scale=0.3, circle, draw,fill=black]{};
\draw (1,0) node [scale=0.3, circle, draw,fill=black]{};
\node [left] at (0,1){{\small $1$}};
\node [right] at (1.1,2.1){{\small $3$}};
\node [right] at (2,1){{\small $4$}};
\node [right] at (1.1,-0.2){{\small $2$}};
\end{tikzpicture}

}\end{minipage}
we have that $a(n)$ is given by the sequence $A165546$ in \cite{oeis}.
\end{thm}

\begin{proof} To avoid $p$ is the same as to avoid the patterns 2143 and 3142, simultaneously. The complement of these is the patterns in \cite[A165546]{oeis}. \end{proof}

\section{POPs of length 5}\label{POP5-sec}

The next theorem gives a new enumerative result on permutations avoiding simultaneously 60 patterns of length 5 in which the first element is larger than the last one. 

\begin{thm}\label{thm-A2} 
For the POP $p=$ \hspace{-3.5mm}
\begin{minipage}[c]{4.8em}\scalebox{1}{
\begin{tikzpicture}[scale=0.4]
\draw [line width=1](0,0)--(0,1);
\draw (0,0) node [scale=0.3, circle, draw,fill=black]{};
\draw (0,1) node [scale=0.3, circle, draw,fill=black]{};
\draw (1,0) node [scale=0.3, circle, draw,fill=black]{};
\draw (2,0) node [scale=0.3, circle, draw,fill=black]{};
\draw (3,0) node [scale=0.3, circle, draw,fill=black]{};
\node [below] at (0,0){\small$5$};
\node [left] at (0,1.2){\small$1$};
\node [below] at (1,0){\small$2$};
\node [below] at (2,0){\small$3$};
\node [below] at (3,0){\small$4$};
\end{tikzpicture}

}\end{minipage}
we have
$$\sum_{n\geq 0}a(n)x^n=\frac{1-x^2}{1-x-2x^2-2x^3-12x^4-8x^5+2x^6+5x^7+x^8}.$$
This is the sequence $A276838$ in \cite{oeis}.
\end{thm}

\begin{proof} 
The case $k=5$ in Theorem~\ref{thm-B6} gives a one-to-one correspondence between $p$-avoiding $n$-permutations and the objects appearing in \cite[A276838]{oeis} with a known g.f..
\end{proof}

The next theorem gives a new enumerative result on permutations avoiding simultaneously 60 patterns of length 5 in which the first element is larger than the second one.

\begin{thm}\label{thm-A1} 
For the POP $p=$ \hspace{-3.5mm}
\begin{minipage}[c]{4.8em}\scalebox{1}{
\begin{tikzpicture}[scale=0.4]
\draw [line width=1](0,0)--(0,1);
\draw (0,0) node [scale=0.3, circle, draw,fill=black]{};
\draw (0,1) node [scale=0.3, circle, draw,fill=black]{};
\draw (1,0) node [scale=0.3, circle, draw,fill=black]{};
\draw (2,0) node [scale=0.3, circle, draw,fill=black]{};
\draw (3,0) node [scale=0.3, circle, draw,fill=black]{};
\node [below] at (0,0){\small$2$};
\node [left] at (0,1.2){\small$1$};
\node [below] at (1,0){\small$3$};
\node [below] at (2,0){\small$4$};
\node [below] at (3,0){\small$5$};
\end{tikzpicture}
}\end{minipage}
we have
$$a(n)=\left\{ \begin{array}{ll} 
n! & \mbox{if } n<5\\ 
n(n-1)(n-2) & \mbox{if } n\geq 5.\end{array}\right.$$
Also, $$\sum_{n\geq 0}a(n)x^n=\frac{1-3x+4x^2+9x^4-7x^5+2x^6}{(1-x)^4}.$$ 
This is essentially  $A007531$ in \cite{oeis} ($a(0)=a(1)=a(2)=0$  in $A007531$).
\end{thm}

\begin{proof} This is an immediate corollary of Theorem~\ref{thm-B3}, or Theorem~\ref{thm-B5}, with $k=5$, $s=3$ and $b(n)=1$ for $n\geq 0$ because only increasing permutations avoid the pattern $p_1=$  \hspace{-3mm} \begin{minipage}[c]{1.5em}\scalebox{1}{
\begin{tikzpicture}[scale=0.3]
\draw (0,0) node {};
\draw [line width=1](0,0)--(0,1);
\draw (0,0) node [scale=0.3, circle, draw,fill=black]{};
\draw (0,1) node [scale=0.3, circle, draw,fill=black]{};
\node [right] at (0,-0.2){\small$2$};
\node [right] at (0,1.2){\small$1$};
\end{tikzpicture}  
}\end{minipage}
. The g.f. can now be easily derived from the recurrence relation. \end{proof}

The next theorem gives a new enumerative result on permutations avoiding simultaneously 24 patterns of length 5 in which the first element is larger than any other element.

\begin{thm}\label{thm-A5} 
For the POP $p=$ \hspace{-3.5mm}
\begin{minipage}[c]{4.2em}\scalebox{1}{
\begin{tikzpicture}[scale=0.4]
\draw [line width=1](0,0)--(1.5,1)--(1,0);
\draw [line width=1](1.5,1)--(2,0);
\draw [line width=1](1.5,1)--(3,0);
\draw (0,0) node [scale=0.3, circle, draw,fill=black]{};
\draw (1.5,1) node [scale=0.3, circle, draw,fill=black]{};
\draw (2,0) node [scale=0.3, circle, draw,fill=black]{};
\draw (1,0) node [scale=0.3, circle, draw,fill=black]{};
\draw (3,0) node [scale=0.3, circle, draw,fill=black]{};
\node [below] at (0,0){\small$2$};
\node [right] at (1.5,1.2){\small$1$};
\node [below] at (2,0){\small$4$};
\node [below] at (1,0){\small$3$};
\node [below] at (3,0){\small$5$};
\end{tikzpicture}
}\end{minipage}
we have
$$a(n)=\left\{ \begin{array}{ll} 
n! & \mbox{if }n<5\\ 
6\cdot 4^{n-3} & \mbox{if } n\geq 5.\end{array}\right.$$
Also,
$$\sum_{n\geq 0}a(n)x^n=\frac{1-3x-2x^2-2x^3}{1-4x}.$$
This is the sequence $A084509$ in \cite{oeis}.
\end{thm}

\begin{proof} This is an immediate corollary of Theorem~\ref{thm-B1}. \end{proof}

The next theorem gives a new enumerative result on permutations avoiding simultaneously 12 patterns of length 5 in which the first and the last elements are larger than any other element.  

\begin{thm}\label{thm-A12} 
For the POP $p=$ \hspace{-3.5mm}
\begin{minipage}[c]{4em}\scalebox{1}{
\begin{tikzpicture}[scale=0.4]
\draw [line width=1](0,0)--(1,1)--(1.5,0);
\draw [line width=1](1,1)--(3,0);
\draw [line width=1](0,0)--(2,1)--(1.5,0);
\draw [line width=1](2,1)--(3,0);
\draw (0,0) node [scale=0.3, circle, draw,fill=black]{};
\draw (1.5,0) node [scale=0.3, circle, draw,fill=black]{};
\draw (3,0) node [scale=0.3, circle, draw,fill=black]{};
\draw (1,1) node [scale=0.3, circle, draw,fill=black]{};
\draw (2,1) node [scale=0.3, circle, draw,fill=black]{};
\node [below] at (0,0){\small$2$};
\node [below] at (1.5,0){\small$3$};
\node [below] at (3,0){\small$4$};
\node [above] at (1,1){\small$1$};
\node [above] at (2,1){\small$5$};
\end{tikzpicture}
}\end{minipage}
we have
$$a(n)=\left\{ \begin{array}{ll} 
n! & \mbox{if }n<5\\ 
6 a(n-1) - 6 a(n-2) & \mbox{if } n\geq 5.\end{array}\right.$$
Also, 
$$\sum_{n\geq 0}a(n)x^n=\frac{1-5x+2x^2}{1-6x + 6x^2}.$$
This is the sequence $A094433$ in \cite{oeis}.
\end{thm}

\begin{proof} This is an immediate corollary of Theorem~\ref{thm-B2}. \end{proof}

The next theorem gives a new enumerative result on permutations avoiding simultaneously 20 patterns of length 5.

\begin{thm}\label{thm-A10} 
For the POP $p=$ \hspace{-3.5mm}
\begin{minipage}[c]{3.2em}\scalebox{1}{
\begin{tikzpicture}[scale=0.4]
\draw [line width=1](0,0)--(0,1)--(1,0)--(1,1)--(0,0);
\draw (0,0) node [scale=0.3, circle, draw,fill=black]{};
\draw (0,1) node [scale=0.3, circle, draw,fill=black]{};
\draw (1,0) node [scale=0.3, circle, draw,fill=black]{};
\draw (1,1) node [scale=0.3, circle, draw,fill=black]{};
\draw (2,0) node [scale=0.3, circle, draw,fill=black]{};
\node [below] at (0,0){\small$2$};
\node [above] at (0,1){\small$1$};
\node [below] at (1,0){\small$3$};
\node [above] at (1,1){\small$4$};
\node [below] at (2,0){\small$5$};
\end{tikzpicture}
}\end{minipage}
we have $$A(x):=\sum_{n\geq 0}a(n)x^n=\frac{1-7x+14x^2-6x^3+4x^4}{(1-4x + 2x^2)^2}.$$
This is the sequence $A094012$ in \cite{oeis}.
\end{thm}

\begin{proof} From Theorem~\ref{thm-B3},
$$a(n)=\left\{ \begin{array}{ll} 
n! & \mbox{if } n<5\\ 
n\cdot b(n-1) & \mbox{if } n\geq 5\end{array}\right.$$
where $b(n)$, along with its g.f.  $B(x):=\sum_{n\geq 0}b(n)x^n=\frac{1-3x}{1-4x + 2x^2}$, is given by Theorem~\ref{thm-10}. Thus, 
$$A(x)=x^2B'(x)+xB(x)+1.$$
Substituting $B(x)$ into the equation above, we obtain the desired result. 
\end{proof}

\begin{thm}\label{thm-A4} 
For the POP $p=$ \hspace{-3.5mm}
\begin{minipage}[c]{3.2em}\scalebox{1}{
\begin{tikzpicture}[scale=0.4]
\draw [line width=1](0,0)--(0,1)--(0,2)--(0,3);
\draw (0,0) node [scale=0.3, circle, draw,fill=black]{};
\draw (0,1) node [scale=0.3, circle, draw,fill=black]{};
\draw (1,0) node [scale=0.3, circle, draw,fill=black]{};
\draw (0,2) node [scale=0.3, circle, draw,fill=black]{};
\draw (0,3) node [scale=0.3, circle, draw,fill=black]{};
\node [left] at (0,0){\small$4$};
\node [left] at (0,1){\small$3$};
\node [right] at (1,0){\small$5$};
\node [left] at (0,2){\small$2$};
\node [left] at (0,3){\small$1$};
\end{tikzpicture}
}\end{minipage}
we have $a(0)=1$ and, for $n\geq 1$, 
$$a(n)=\frac{1}{n(n+1)}\sum_{i=0}^{n-1}{2i \choose i}{n \choose i+1}{n+1 \choose i+1}.$$
This is the sequence $A128088$ in \cite{oeis}.
\end{thm}

\begin{proof} It is known \cite{BousqMe02-03} that the number of 1234-avoiding $n$-permutations is
$$\frac{1}{(n+1)^2(n+2)}\sum_{i=0}^{n}{2i \choose k}{n+1 \choose i+1}{n+2 \choose i+1}.$$
We can now use Theorem~\ref{thm-B3} to obtain the desired result.\end{proof}

\section{Concluding remarks}\label{final-sec}

A number of conjectured connections between sequences in the OEIS and permutations avoiding POPs of length 5 appear in Table~\ref{tab-pop-5-conj}. 

\begin{table}[!ht]
\begin{center}
\begin{tabular}{c|c|l}
\hline
\rowcolor{gray!20!}
{\bf POP}  &  {\bf OEIS} &  {\bf Equinumerous structures} \\ 
\multirow{2}{*}{ 
\begin{tikzpicture}[scale=0.3]
\draw [line width=1](0,0)--(1,1)--(2,0);
\draw (0,0) node [scale=0.3, circle, draw,fill=black]{};
\draw (1,1) node [scale=0.3, circle, draw,fill=black]{};
\draw (2,0) node [scale=0.3, circle, draw,fill=black]{};
\draw (3,0) node [scale=0.3, circle, draw,fill=black]{};
\node [left] at (0,-0.2){\small$2$};
\node [right] at (1,1.2){\small$1$};
\node [left] at (2,-0.2){\small$4$};
\node [right] at (3,-0.2){\small$3$};
\end{tikzpicture}
}
& A111281 & permutations avoiding the patterns \\
& & 2413, 2431, 4213, 3412, 3421, 4231, 4321, 4312\\
\hline
\multirow{2}{*}{ 
\begin{tikzpicture}[scale=0.3]
\draw [line width=1](0,0)--(1,1)--(2,0);
\draw (0,0) node [scale=0.3, circle, draw,fill=black]{};
\draw (1,1) node [scale=0.3, circle, draw,fill=black]{};
\draw (2,0) node [scale=0.3, circle, draw,fill=black]{};
\draw (3,0) node [scale=0.3, circle, draw,fill=black]{};
\node [left] at (0,-0.2){\small$1$};
\node [right] at (1,1.2){\small$2$};
\node [left] at (2,-0.2){\small$4$};
\node [right] at (3,-0.2){\small$3$};
\end{tikzpicture}
}
& A111282 & permutations avoiding the patterns \\
& & 1432, 2431, 3412, 3421, 4132, 4231, 4312, 4321\\
\hline
\multirow{2}{*}{ 
\begin{tikzpicture}[scale=0.3]
\draw [line width=1](0,0)--(0,1)--(1,0)--(1,1);
\draw (0,0) node [scale=0.3, circle, draw,fill=black]{};
\draw (0,1) node [scale=0.3, circle, draw,fill=black]{};
\draw (1,0) node [scale=0.3, circle, draw,fill=black]{};
\draw (1,1) node [scale=0.3, circle, draw,fill=black]{};
\node [left] at (0,-0.1){\small$2$};
\node [left] at (0,1.1){\small$1$};
\node [right] at (1,-0.1){\small$3$};
\node [right] at (1,1.1){\small$4$};
\end{tikzpicture}
}
& A111277
 & permutations avoiding the patterns 2413, 4213, 2431, \\
& &  4231, 4321; also, permutations avoiding the patterns  \\
& & 3142, 3412, 3421, 4312, 4321\\

\hline
\multirow{2}{*}{ 
\begin{tikzpicture}[scale=0.3]
\draw [line width=1](0,0)--(0,1)--(1,0)--(1,1)--(0,0);
\draw (0,0) node [scale=0.3, circle, draw,fill=black]{};
\draw (0,1) node [scale=0.3, circle, draw,fill=black]{};
\draw (1,0) node [scale=0.3, circle, draw,fill=black]{};
\draw (1,1) node [scale=0.3, circle, draw,fill=black]{};
\node [left] at (0,-0.1){\small$2$};
\node [left] at (0,1.1){\small$1$};
\node [right] at (1,-0.1){\small$3$};
\node [right] at (1,1.1){\small$4$};
\end{tikzpicture}
}
& A006012
 & permutations avoiding the vincular patterns \\
& & $1\underline{32}4$, $1\underline{42}3$, $2\underline{31}4$, $2\underline{41}3$ considered in \cite{B17}   \\

\hline
\multirow{2}{*}{ 
\begin{tikzpicture}[scale=0.3]
\draw [line width=1](0,0)--(1,1)--(2,0);
\draw [line width=1](1,1)--(1,0);
\draw (0,0) node [scale=0.3, circle, draw,fill=black]{};
\draw (1,1) node [scale=0.3, circle, draw,fill=black]{};
\draw (2,0) node [scale=0.3, circle, draw,fill=black]{};
\draw (1,0) node [scale=0.3, circle, draw,fill=black]{};
\node [below] at (0,0){\small$2$};
\node [right] at (1,1.2){\small$1$};
\node [below] at (2,0){\small$4$};
\node [below] at (1,0){\small$3$};
\end{tikzpicture}
}
& A025192
& permutations $\pi_1\cdots\pi_{3n}$ avoiding the patterns 231, \\
& &  312, 321 and satisfying 
$\pi_{3i+1}<\pi_{3i+2}$ and \\
& & $\pi_{3i+1}<\pi_{3i+3}$ for all $0\leq i<n$. Equivalently, 2-ary \\
& & shrub forests of $n$ heaps avoiding the patterns  231, \\
& & 312, 321; see  \cite{BLNPPRT16} \\

\hline

\multirow{2}{*}{ 
\begin{tikzpicture}[scale=0.4]
\draw [line width=1](0,0)--(1,1)--(1,2);
\draw [line width=1](1,0)--(1,1)--(2,0);
\draw (0,0) node [scale=0.3, circle, draw,fill=black]{};
\draw (1,1) node [scale=0.3, circle, draw,fill=black]{};
\draw (1,0) node [scale=0.3, circle, draw,fill=black]{};
\draw (2,0) node [scale=0.3, circle, draw,fill=black]{};
\draw (1,2) node [scale=0.3, circle, draw,fill=black]{};
\node [below] at (0,0){\small$2$};
\node [left] at (1,1){\small$1$};
\node [below] at (1,0){\small$3$};
\node [below] at (2,0){\small$4$};
\node [left] at (1,2){\small$5$};
\end{tikzpicture}
}
& A054872 & permutations avoiding the patterns 12345, 13245,\\
& &  21345, 23145, 31245, 32145; note that avoiding these \\
& &  patterns is the same as avoiding the POP 
$\{5>4,$ \\
& & $4>1,4>2,4>3\}$\\

\hline
\multirow{2}{*}{ 
\begin{tikzpicture}[scale=0.4]
\draw [line width=1](0,0)--(1,1)--(2,2);
\draw [line width=1](0,2)--(1,1)--(2,0);
\draw (0,0) node [scale=0.3, circle, draw,fill=black]{};
\draw (0,2) node [scale=0.3, circle, draw,fill=black]{};
\draw (1,1) node [scale=0.3, circle, draw,fill=black]{};
\draw (2,2) node [scale=0.3, circle, draw,fill=black]{};
\draw (2,0) node [scale=0.3, circle, draw,fill=black]{};
\node [below] at (0,0){\small$3$};
\node [above] at (0,2){\small$1$};
\node [below] at (2,0){\small$4$};
\node [below] at (1,1){\small$5$};
\node [above] at (2,2){\small$2$};
\end{tikzpicture}
}
& A212198
 & permutations avoiding the marked mesh pattern \\
& &  M(2,0,2,0) in \cite{KR12}; these permutations are proved  \\
& &  to be in bijection with  pattern-avoiding involutions \\ 
& & {\bf I}$_n(>,\neq,>)$ in \cite{MS18}   \\
\hline
\multirow{2}{*}{ 
\begin{tikzpicture}[scale=0.4]
\draw [line width=1](0,0)--(1,1)--(1,2)--(1,3);
\draw [line width=1](1,1)--(2,0);
\draw (0,0) node [scale=0.3, circle, draw,fill=black]{};
\draw (1,1) node [scale=0.3, circle, draw,fill=black]{};
\draw (1,3) node [scale=0.3, circle, draw,fill=black]{};
\draw (2,0) node [scale=0.3, circle, draw,fill=black]{};
\draw (1,2) node [scale=0.3, circle, draw,fill=black]{};
\node [below] at (0,0){\small$3$};
\node [left] at (1,1){\small$5$};
\node [left] at (1,3){\small$2$};
\node [below] at (2,0){\small$4$};
\node [left] at (1,2){\small$1$};
\end{tikzpicture}
}
& A224295 & permutations avoiding the patterns 12345 and 12354;\\
& &  note that, by Theorem~\ref{trivial-sym-thm},  avoiding these patterns is \\ 
& & the same as avoiding the POP $\{1>2,2>3,3>4,$ \\
& & $3>5\}$\\
%
\hline
\end{tabular}
\end{center}
\caption{A list of potentially interesting bijective questions for permutations avoiding a POP and other pattern avoiding permutations.}\label{bij-questions-permutations}
\end{table}

One can ask a number of bijective questions even in some cases where connections to the OEIS were explained. In Tables~\ref{bij-questions-permutations} and~\ref{bij-questions} we list a number of potentially interesting bijective questions. We refer to the OEIS \cite{oeis} for the definitions/further details of the objects mentioned in the table. Note that Table~\ref{bij-questions-permutations} contains bijective questions related to previously considered permutation pattern avoidance. In that table, a vincular pattern is like a classical pattern, but it allows imposing the condition on certain elements in an occurrence of the pattern to be consecutive in a permutation, which is denoted by underlying the respective elements in the pattern \cite{Kit5}. For example, in the permutation $2415763$ there are three occurrences of the vincular pattern $1\underline{23}$, namely, the subsequences $257$, $457$ and $157$. Note that, e.g. the subsequence $256$, being an occurrence of the pattern 123, is not an occurrence of the pattern $1\underline{23}$ because $5$ and $6$ are not consecutive.

\begin{table}[!ht]
\begin{center}
\begin{tabular}{c|c|l}
\hline
\rowcolor{gray!20!}
{\bf POP}  &  {\bf OEIS} &  {\bf Equinumerous structures} \\ 
\hline
\multirow{3}{*}{ 
\begin{tikzpicture}[scale=0.3]
\draw [line width=1](0,0)--(0,1);
\draw (0,0) node [scale=0.3, circle, draw,fill=black]{};
\draw (0,1) node [scale=0.3, circle, draw,fill=black]{};
\draw (1,0) node [scale=0.3, circle, draw,fill=black]{};
\draw (2,0) node [scale=0.3, circle, draw,fill=black]{};
\node [below] at (0,0){\small$3$};
\node [left] at (0,1.2){\small$1$};
\node [below] at (1,0){\small$2$};
\node [below] at (2,0){\small$4$};
\end{tikzpicture}
}
& A045925 & levels in all compositions of $n+1$ with only 1's and 2's \\
& & \\
& & \\
\hline
\multirow{2}{*}{ 
\begin{tikzpicture}[scale=0.3]
\draw [line width=1](0,0)--(0,1);
\draw (0,0) node [scale=0.3, circle, draw,fill=black]{};
\draw (0,1) node [scale=0.3, circle, draw,fill=black]{};
\draw (1,0) node [scale=0.3, circle, draw,fill=black]{};
\draw (2,0) node [scale=0.3, circle, draw,fill=black]{};
\node [below] at (0,-0.2){\small$4$};
\node [left] at (0,1.2){\small$1$};
\node [below] at (1,-0.2){\small$2$};
\node [below] at (2,-0.2){\small$3$};
\end{tikzpicture}
}
& A214663 & $n$-permutations for which the partial sums of signed \\ 
& & displacements do not exceed 2 \\
& A232164 & Weyl group elements, not containing an $s_r$ factor, which \\ 
& & contribute nonzero terms to Kostant's weight multiplicity \\ 
& & formula when computing the multiplicity of the \\
& & zero-weight in the adjoint representation for the Lie  \\ 
& & algebra of type $C$ and rank $n$ \\
\hline
\multirow{2}{*}{ 
\begin{tikzpicture}[scale=0.3]
\draw [line width=1](0,0)--(0,1);
\draw [line width=1](1,0)--(1,1);
\draw (0,0) node [scale=0.3, circle, draw,fill=black]{};
\draw (0,1) node [scale=0.3, circle, draw,fill=black]{};
\draw (1,0) node [scale=0.3, circle, draw,fill=black]{};
\draw (1,1) node [scale=0.3, circle, draw,fill=black]{};
\node [left] at (0,-0.2){\small$4$};
\node [left] at (0,1.2){\small$1$};
\node [right] at (1,-0.2){\small$2$};
\node [right] at (1,1.2){\small$3$};
\end{tikzpicture}
}
& A271897 & sum of all second elements at level n of the TRIP-Stern \\
& & sequence corresponding to the permutation triple $(e,e,e)$ \\
\hline
\multirow{2}{*}{ 
\begin{tikzpicture}[scale=0.3]
\draw [line width=1](0,0)--(0,1)--(1,0)--(1,1);
\draw (0,0) node [scale=0.3, circle, draw,fill=black]{};
\draw (0,1) node [scale=0.3, circle, draw,fill=black]{};
\draw (1,0) node [scale=0.3, circle, draw,fill=black]{};
\draw (1,1) node [scale=0.3, circle, draw,fill=black]{};
\node [left] at (0,-0.1){\small$3$};
\node [left] at (0,1.1){\small$1$};
\node [right] at (1,-0.1){\small$2$};
\node [right] at (1,1.1){\small$4$};
\end{tikzpicture}
}
& A052544 & compositions of $3n + 1$ into parts of the form $3m + 1$ \\
& & \\
\hline
\multirow{2}{*}{ 
\begin{tikzpicture}[scale=0.4]
\draw [line width=1](0,0)--(1.5,1)--(1,0);
\draw [line width=1](1.5,1)--(2,0);
\draw [line width=1](1.5,1)--(3,0);
\draw (0,0) node [scale=0.3, circle, draw,fill=black]{};
\draw (1.5,1) node [scale=0.3, circle, draw,fill=black]{};
\draw (2,0) node [scale=0.3, circle, draw,fill=black]{};
\draw (1,0) node [scale=0.3, circle, draw,fill=black]{};
\draw (3,0) node [scale=0.3, circle, draw,fill=black]{};
\node [below] at (0,0){\small$2$};
\node [right] at (1.5,1.2){\small$1$};
\node [below] at (2,0){\small$4$};
\node [below] at (1,0){\small$3$};
\node [below] at (3,0){\small$5$};
\end{tikzpicture}
}
& A084509 & number of ground-state 3-ball juggling sequences of  \\
& &period $n$ \\[3mm]
\hline
\multirow{2}{*}{ 
\begin{tikzpicture}[scale=0.4]
\draw [line width=1](0,1)--(1,2)--(2,1)--(1,0);
\draw [line width=1](2,1)--(3,0);
\draw (0,1) node [scale=0.3, circle, draw,fill=black]{};
\draw (2,1) node [scale=0.3, circle, draw,fill=black]{};
\draw (1,0) node [scale=0.3, circle, draw,fill=black]{};
\draw (3,0) node [scale=0.3, circle, draw,fill=black]{};
\draw (1,2) node [scale=0.3, circle, draw,fill=black]{};
\node [below] at (0,1){\small$2$};
\node [right] at (2,1){\small$3$};
\node [below] at (1,0){\small$4$};
\node [below] at (3,0){\small$5$};
\node [right] at (1,2){\small$1$};
\end{tikzpicture}
}
& A118376 & series-reduced enriched plane trees of weight $n$;  also,  \\
& & trees of weight $n$, where nodes have positive integer\\
& &  weights and the sum of the weights of the children \\ 
& & of a node is equal to the weight of the node\\
\hline
\end{tabular}
\end{center}
\caption{A list of potentially interesting bijective questions for permutations avoiding a POP and other combinatorial structures.}\label{bij-questions}
\end{table}

As a final remark, we note that there is a simple connection between  $p$-avoiding $n$-permutations for the POP $p=$ \hspace{-3.5mm}
\begin{minipage}[c]{4.8em}\scalebox{1}{
\begin{tikzpicture}[scale=0.3]
\draw [line width=1](0,0)--(1,1)--(2,0);
\draw (0,0) node [scale=0.3, circle, draw,fill=black]{};
\draw (1,1) node [scale=0.3, circle, draw,fill=black]{};
\draw (2,0) node [scale=0.3, circle, draw,fill=black]{};
\draw (3,0) node [scale=0.3, circle, draw,fill=black]{};
\node [left] at (0,-0.2){\small$2$};
\node [right] at (1,1.2){\small$1$};
\node [left] at (2,-0.2){\small$3$};
\node [right] at (3,-0.2){\small$4$};
\end{tikzpicture}
}\end{minipage} considered in Theorem~\ref{thm-16} and $n$-permutation in the class $S_2(1,n)$ in \cite{BurKit} dealing with the vincular POP \hspace{-3.5mm}
\begin{minipage}[c]{4.8em}\scalebox{1}{
\begin{tikzpicture}[scale=0.3]
\draw [line width=1](0,0)--(1,1)--(2,0);
\draw (0,0) node [scale=0.3, circle, draw,fill=black]{};
\draw (1,1) node [scale=0.3, circle, draw,fill=black]{};
\draw (2,0) node [scale=0.3, circle, draw,fill=black]{};
\draw (3,0) node [scale=0.3, circle, draw,fill=black]{};
\node [left] at (0,-0.2){\small$2$};
\node [right] at (1,1.2){\small$3$};
\node [left] at (2,-0.2){\small$4$};
\node [right] at (3,-0.2){\small$1$};
\end{tikzpicture}
}\end{minipage} in occurrences of which the elements in all but the first position must be consecutive. Wilf-equivalence of these POPs is essentially given by Wilf-equivalence of the sets of patterns $\{\underline{132},\underline{231}\}$ and $\{312, 321\}$ \cite{Kit5}. 

\section*{Acknowledgments}
The first author was supported in part by the Fundamental Research Funds for the Central Universities (31020170QD101) and the National Science Foundation of China (No. 11801447).  Also, the authors are grateful to Stephen Gardiner for producing his software.

\end{document}